\documentclass[11pt]{amsart}
\usepackage{amsmath,amssymb,amsthm,graphicx, url, verbatim, enumerate, float, amsfonts, relsize}
\usepackage[margin=1in]{geometry}
\usepackage{tikz-cd}
\usetikzlibrary{matrix,arrows}
\usepackage{color}
\theoremstyle{definition}
\newtheorem{defn}{Definition}[section]

\theoremstyle{plain}
\newtheorem{thm}[defn]{Theorem}
\newtheorem{conj}[defn]{Conjecture}
\newtheorem{prop}[defn]{Proposition}
\newtheorem{lem}[defn]{Lemma} 
\newtheorem{cor}[defn]{Corollary}

\newcommand{\Z}{\mathbb{Z}}

\newcommand{\Kh}{H^{Kh}}
\newcommand{\Ckh}{C^{Kh}}
\newcommand{\Sk}{\mathcal{S}}
\newcommand{\jwproj}{\vcenter{\hbox{\def\svgwidth{.020\columnwidth} 
\begingroup%
  \makeatletter%
  \providecommand\color[2][]{%
    \errmessage{(Inkscape) Color is used for the text in Inkscape, but the package 'color.sty' is not loaded}%
    \renewcommand\color[2][]{}%
  }%
  \providecommand\transparent[1]{%
    \errmessage{(Inkscape) Transparency is used (non-zero) for the text in Inkscape, but the package 'transparent.sty' is not loaded}%
    \renewcommand\transparent[1]{}%
  }%
  \providecommand\rotatebox[2]{#2}%
  \ifx\svgwidth\undefined%
    \setlength{\unitlength}{69.35bp}%
    \ifx\svgscale\undefined%
      \relax%
    \else%
      \setlength{\unitlength}{\unitlength * \real{\svgscale}}%
    \fi%
  \else%
    \setlength{\unitlength}{\svgwidth}%
  \fi%
  \global\let\svgwidth\undefined%
  \global\let\svgscale\undefined%
  \makeatother%
  \begin{picture}(1,1.38428262)%
    \put(0,0){\includegraphics[width=\unitlength]{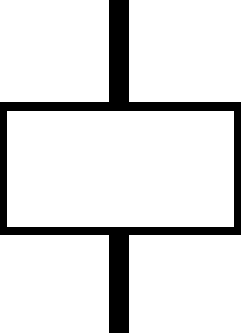}}%
  \end{picture}%
\endgroup%
}}}
\newcommand{\jwprojp}{\vcenter{\hbox{\def\svgwidth{.020\columnwidth} 
\begingroup%
  \makeatletter%
  \providecommand\color[2][]{%
    \errmessage{(Inkscape) Color is used for the text in Inkscape, but the package 'color.sty' is not loaded}%
    \renewcommand\color[2][]{}%
  }%
  \providecommand\transparent[1]{%
    \errmessage{(Inkscape) Transparency is used (non-zero) for the text in Inkscape, but the package 'transparent.sty' is not loaded}%
    \renewcommand\transparent[1]{}%
  }%
  \providecommand\rotatebox[2]{#2}%
  \ifx\svgwidth\undefined%
    \setlength{\unitlength}{70.15765381bp}%
    \ifx\svgscale\undefined%
      \relax%
    \else%
      \setlength{\unitlength}{\unitlength * \real{\svgscale}}%
    \fi%
  \else%
    \setlength{\unitlength}{\svgwidth}%
  \fi%
  \global\let\svgwidth\undefined%
  \global\let\svgscale\undefined%
  \makeatother%
  \begin{picture}(1,1.44816701)%
    \put(0,0){\includegraphics[width=\unitlength]{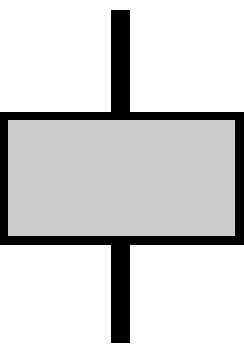}}%
  \end{picture}%
\endgroup%
}}}

\newcommand{\sgn}{\text{sgn}}

\newcommand{\db}[1]{[\![ #1 ]\!]}
\title{A trivial tail homology for non $A$-adequate links}
\author[C. Lee]{Christine Ruey Shan Lee}
\address[]{Department of Mathematics, University of Texas at Austin, Austin TX 78712}
\email[]{clee@math.utexas.edu}
\thanks{Lee was supported in part by NSF grant MSPRF-DMS 1502860.}

\begin{document}

\begin{abstract} We prove a conjecture of Rozansky's concerning his categorification of the tail of the colored Jones polynomial for an $A$-adequate link. We show that the tail homology groups he constructs are trivial for non $A$-adequate links. 
\end{abstract}

\maketitle
\tableofcontents

\section{Introduction}

Let $D$ be a diagram of a link $K$ in $S^3$. A \emph{Kauffman state} is a choice of replacing every crossing of $D$ by the $A$- or $B$-resolution as in Figure \ref{fig:abres}, with the dashed segment recording the location of the crossing before the replacement.

\begin{figure}[ht]
\centering
\includegraphics{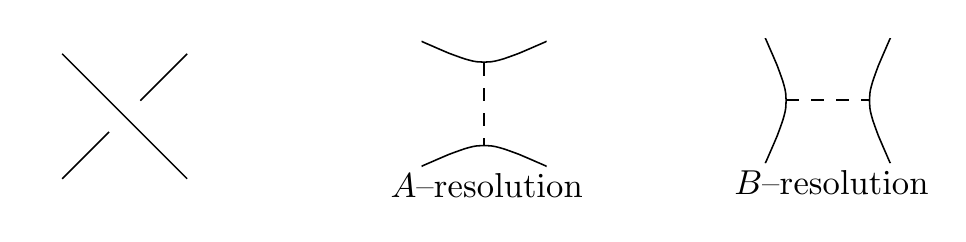}
\caption{$A$- and $B$-resolutions of a crossing.}
\label{fig:abres}
\end{figure} 

Applying a Kauffman state results in a set of disjoint circles called \emph{state circles}. We form a $\sigma$-\emph{state graph} $s_{\sigma}(D)$ for each Kauffman state $\sigma$ by letting the resulting state circles be vertices and the segments be edges. The \emph{all-$A$} state graph $s_A(D)$ comes from the Kauffman state which chooses the $A$ resolution at every crossing of $D$. 

\begin{defn}\label{defn:adequate-diagram} A link diagram $D$ is \emph{$A$-adequate} if its all-$A$ state graph has no one-edged loops. 
\end{defn} 

\begin{figure}[H]
\def\svgwidth{.6\columnwidth}
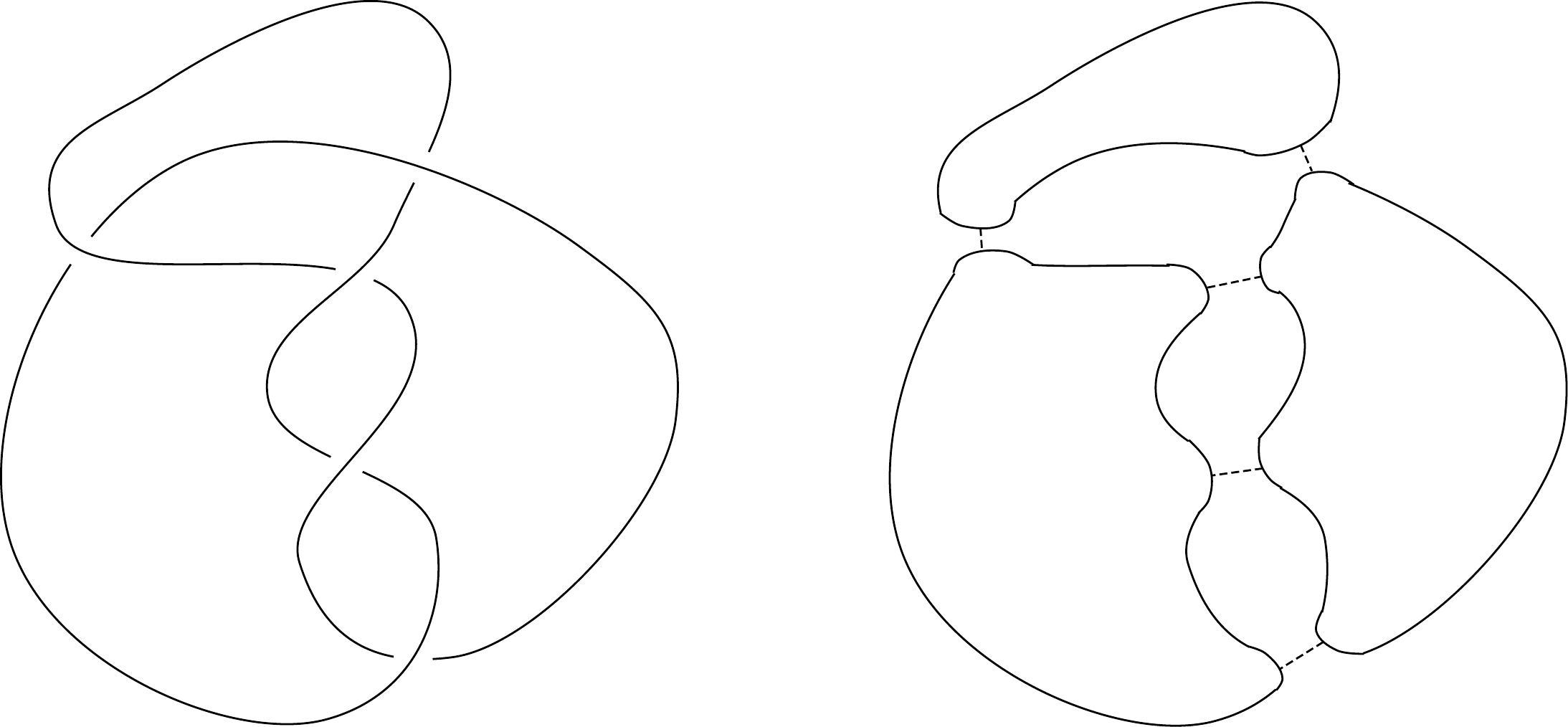
\caption{An adequate diagram and its all-$A$ state.}
\end{figure}

\begin{figure}[H]
\def\svgwidth{.6\columnwidth}
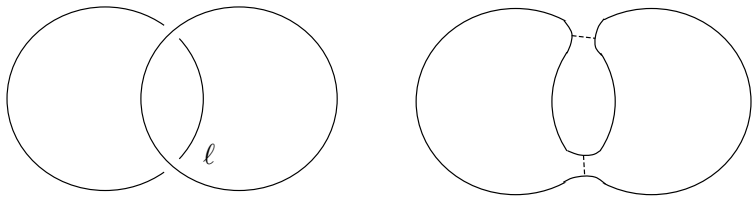
\caption{A non $A$-adequate diagram of the two component unlink. A crossing corresponding to a one-edged loop in $s_A(D)$ is marked $\ell$.}
\end{figure}

A link $K$ is said to be \emph{$A$-adequate} if it admits an $A$-adequate diagram. Every link diagram of a non $A$-adequate link is not $A$-adequate. A diagram is \emph{$B$-adequate} if its mirror image is $A$-adequate.  
\begin{defn} \label{defn:adequate} A link $K$ is \emph{semi-adequate} (\emph{$A$- or $B$-adequate}) if it admits a diagram that is $A$- or $B$-adequate. 
\end{defn} 

First studied by Lickorish and Thistlethwaite \cite{LT88}, \emph{semi-adequate} links form a rich class of links which includes alternating links. Denote the colored Jones polynomial of a link $K\subset S^3$ by $\{J_K(q; n)\}_{n=1}^{\infty}$, where $J_K(q;n) \in \mathbb{Z}[q, q^{-1}]$, and $J_K(q;1)$ is the unreduced Jones polynomial of $K$. See Definition \ref{defn:cp} for our normalization convention. Armond \cite{Arm} and Garoufalidis and Le \cite{GL15} have independently shown the following result. It states a stability property, first conjectured in \cite{DL06} with partial evidence, of the colored Jones polynomial of a semi-adequate link. 

Let $d(n)$ be the minimum degree of the $n$th colored Jones polynomial $J_K(q;n)$.
\begin{thm}[\cite{Arm, GL15}] \label{thm:stab}
 For $i > 1$, let $\beta_i$ be the coefficient of $q^{d(i)+2(i-2)}$ of $J_K(q;i)$. If $K$ admits an $A$-adequate diagram, then the coefficient of $q^{d(n)+2(i-2)}$  of $J_K(q;n)$ is equal to $\beta_i$ for all $n\geq i$.     
\end{thm}  

For an $A$-adequate link $K$ one defines a power series 
\begin{equation} T_K(q) = \sum_{i=2}^{\infty} \beta_iq^{2(i-2)}, \label{eq:tail} \end{equation} called a \emph{tail} of the colored Jones polynomial of $K$. For a $B$-adequate link with $B$-adequate diagram $D$, the mirror image $\overline{D}$ is $A$-adequate, and we may apply Theorem \ref{thm:stab} to obtain a \emph{head} of the colored Jones polynomial, since $J_K(q; n) = J_{\overline{K}}(q^{-1}; n)$.   

Subsequent to \cite{Arm, GL15}, Rozansky \cite{Roz12} has shown that stability behaviors also exist in the categorification of the colored Jones polynomial. More precisely, let $\{ H_{i,j}^{Kh}(D, n) \}$ be the set of bi-graded chain groups constructed by Rozansky \cite{Roz10} which categorifies the colored Jones polynomial, i.e., 
\begin{equation} J_K(q; n)= ((-1)^nq^{\frac{n^2+2n}{2}})^{\omega(D)}\sum_{i, j}(-1)^j q^{i+j}\text{ dim } \Kh_{i, j}(D, n). \label{eq:echar} \end{equation}
See \cite{FKS06, CK12} for other and previous constructions of the categorification of the polynomial. In Section \ref{subsec:notkbracket}, we describe the convention used in this paper for the categorification of the colored Jones polynomial adapted from \cite{Roz12}. In \cite{Roz12}, Rozansky studies a shifted version of the categorifying chain groups, $\{ \widetilde{\Kh_{i, j}}(D, n) \}$, defined as follows.   
\[ \widetilde{\Kh_{i, j}}(D, n) := \textbf{h}^{\frac{1}{2}n^2c(D)} \textbf{q}^{n|s_A(D)|} \Kh_{i, j}(D, n), \]
where $c(D)$ is the number of crossings of $D$, $|s_A(D)|$ is the number of state circles in the all-$A$ state of $D$, and $\textbf{h}$, $\textbf{q}$ indicate the shifts to the homological and quantum grading, respectively. See Section \ref{subsec:shifted} for the equivalence of this definition to his. He shows that there exists a directed system of degree-preserving maps 
\[\widetilde{\Kh}(D, n) \stackrel{f_n}{\longrightarrow} \widetilde{\Kh}(D, n+1), \] where $f_n$ are isomorphisms on $\widetilde{\Kh_{i,*}}(D, n)$ for $i\leq n-1$, see Section \ref{sec:tailhomology} for details. This implies the existence of a \emph{tail homology} $H^{\infty}(D)$, which is defined as the direct limit of the directed system. 

Let 
\[ J_{D, \infty}(q):= 
\sum_{i, j} (-1)^j q^{i+j} \dim H_{i,j}^{\infty}(D) \] be the graded Euler characteristic of $H^{\infty}(D)$. For $A$-adequate links, Rozansky shows that the tail homology categorifies $T_K(q)$ since $J_{D, \infty}(q)$ determines the lower powers of $J_K(q; n)$ for $D$ an $A$-adequate diagram of $K$. See Theorem \ref{thm:adequatelower} in Section \ref{subsec:tail} for the precise statement.

He makes the following conjecture, which we adjust for the convention used in this paper, regarding the directed system thus constructed.

\begin{conj}[{\cite[Conjecture 2.14]{Roz12}}] \label{thm:main} If a diagram $D$ is not $A$-adequate, then $H^{\infty}(D)$ is trivial. 
\end{conj}

We are motivated by the following result which provides partial evidence to Conjecture \ref{thm:main}.

Let
\begin{equation}  h_n(D)= -\frac{n^2}{2}c(D) - n|s_{A}(D)| + \omega(D)\frac{n^2+2n}{2} \label{eq:upperbound}, \end{equation} where  $\omega(D)$ is the writhe of the diagram.

\begin{thm}[\cite{Lee14}] \label{thm:tail} Suppose that a link diagram $D$ is not $A$-adequate, then
\[ d(n) \geq h_n(D) + 2(n-1) \text{ for } n > 1. \]  
\end{thm}
Let $h(n)$ be the maximum of $h_n(D)$ taken over all diagrams $D$ of a link. Theorem \ref{thm:tail} shows that the difference between the lower bound $h(n)$ and the actual degree of the $n$th colored Jones polynomial increases linearly with $n$, which would be a consequence of Conjecture \ref{thm:main} for large $n$. 

\subsection{Main result}
In this paper we prove Conjecture \ref{thm:main}, and therefore, the categorical analogue of Theorem \ref{thm:tail}. The result sheds some light on the behavior of the categorification of the colored Jones polynomial when the link is not $A$-adequate. Prior to Theorem \ref{thm:tail}, it was not known whether a link may still achieve the lower bound of $h(n)$ when it is not $A$-adequate. The result of this paper shows that the categorification has a similar gap, linear with respect to $n$, between the lower bound $-\frac{n^2}{2}c(D)$ of the minimum homological degree of the complex, and the actual minimum homological degree. 

The precise implication for the colored Jones polynomial is given by the following corollary.
\begin{cor} \label{cor:nacjp} 
The graded Euler characteristic of the tail homology determines the lower powers of $q$ in the unicolored Jones polynomial of a non $A$-adequate link $K$: 
\[J_K(q; n)=\left((-1)^nq^{\frac{n^2+2n}{2}}\right)^{\omega(D)} q^{-\frac{1}{2}n^2c(D)-n|s_A(D)|}\left(J_{D, \infty}(q) + O(q^{\frac{1}{2}n-\frac{1}{2}c(D)-\frac{3}{2}c(D)^{\ell}}) \right), \] where $c(D)^{\ell}$ is the number of non $A$-adequate crossings of $D$ whose corresponding segments in $s_A(D)$ are one-edged loops.
If $D$ is a diagram realizing the lower bound 
\[ h(n) := \max_{D \text{ of }K}\left\{ -\frac{n^2}{2}c(D)-n|s_A(D)|+ \omega(D)\frac{n^2+2n}{2}\right\}, \] then 
\[J_K(q; n)=(-1)^{n(\omega(D))}q^{h(n)}\left(J_{D, \infty}(q) + O(q^{\frac{1}{2}n-\frac{1}{2}c(D)-\frac{3}{2}c(D)^{\ell}}) \right).\] 
\end{cor}

This is a straightforward consequence of the general homological bounds in Theorem \ref{thm:bounds} by Rozansky. Theorem \ref{thm:main} then implies that $J_{D, \infty}(q) = 0$ for a non $A$-adequate link, and we recover Theorem \ref{thm:tail} in the sense that the gap between the actual degree $d(n)$ of the polynomial and the lower bound $h_n(D)$ increases linearly with respect to $n$ for sufficiently large $n$. 

Given a non-$A$ adequate diagram, which necessarily has a crossing $\ell$ corresponding to a one-edged loop in $s_A(D)$, the key to the argument is localizing the behavior of the tail homology to the categorification of a skein $\Sk^k$ containing $\ell$ shown below in Figure \ref{fig:gformp}. With the help of the homological machinery developed by Rozansky in \cite{Roz12}, the homology of this skein is approximated by that of the unknot with a left-hand twist, which has trivial homology in homological gradings $\leq n-1$.  

\begin{figure}[H]
\centering
    \def\svgwidth{.3\columnwidth}
    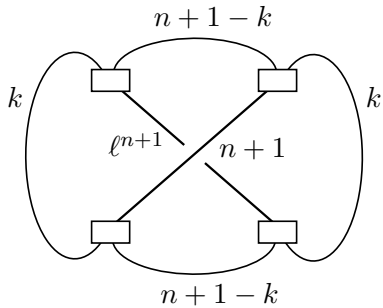
    \caption{\label{fig:gformp} The skein $\Sk^k$. The notation $\ell^{n+1}$ indicates the set of $(n+1)^2$ crossings corresponding to $\ell$ in the $n+1$-blackboard cable of $D$.}
\end{figure}

\subsection{Relation to a general tail}
As for the question of stability behaviors for all links along the lines of Theorem \ref{thm:stab}, a tail may be defined for a general link, if it exists, as follows from \cite{AD2} with our conventions.  
\begin{defn}For a Laurent polynomial $P_1(q)$ and a power series $P_2(q)$ we define 
\[P_1(q) \doteq_{2n} P_2(q) \] if  $P_1(q)$ coincides with $P_2(q) \pmod{ q^{2n}}$ up to multiplication with $\pm q^{s}$, for $s\in \Z$.  
\end{defn}  

\begin{defn} The \emph{tail} of the colored Jones polynomial of a knot $K$, if it exists, is a series $T_K(q) = \sum_{j=2}^{\infty}  a_jq^{2(j-2)}$ with 
\[J_K(q; n) \doteq_{2n} T_K(q) \text{\ for all $n$}. \]
\end{defn} 

The existence and behavior of a tail $T_K(q)$ for a general link remains an interesting and important question. It is the hope that the techniques of this paper may be applied to the study of a categorification of $T_K(q)$. Note that Conjecture \ref{thm:main} and Theorem \ref{thm:tail} do not provide information on the existence or the nature of $T_K(q)$, only that if $T_K(q)$ exists, it must occur at a distance, linear with respect to $n$, from the diagrammatic lower bound. For understanding $T_K(q)$, a persistent difficulty lies in determining $d(n)$ when a link is not $A$-adequate. In the case of torus knots where the colored Jones polynomial is explicitly computed, it has been shown that a tail does not exist, rather, multiple tails exist which indicates a more complex stability behavior of the polynomial \cite{AD2}. Although this is also expected to be reflected in the categorification of the colored Jones polynomial of torus links, a potentially simpler class to study would be non $A$-adequate links which are still expected to admit a single tail. In \cite{LV}, the author and Roland van der Veen have determined $d(n)$ for many non $A$-adequate 3-string pretzel knots. It is expected that a single tail exists for these knots, and we will address this topic for the polynomial and its categorification in the future using the techniques in this paper.

\subsection{Organization}
 Section \ref{sec:prelim} and Section \ref{sec:tailhomology} gather the necessary definitions for understanding the main result and the convention for the colored Jones polynomial and categorification used in this paper. In Section \ref{sec:tailhomology}, we describe Rozansky's definition of a tail homology. Conjecture \ref{thm:main} is proven in Section \ref{sec:conj} as Theorem \ref{cor:tail-trivial}.

\subsection{Acknowledgements}
The author would like to thank the referee for the careful reading  of this paper and many helpful comments which have greatly contributed to the quality of the paper in its final form. The author would also like to thank the organizers for the Quantum Topology and Hyperbolic Geometry Conference in Nha Trang, Vietnam, where the author first heard of Rozansky's work on the tail homology of the colored Jones polynomial.

\section{Background} \label{sec:prelim}

\subsection{Skein theory} \label{subsec:skein}
We will follow \cite{Lic97} in defining the Temperley-Lieb algebra, except he uses $A$, and we make the substitution of variable $A^2=q^{-1}$. This is for the convenience of not having to substitute $q$ for $A$ later and to avoid confusion with the ``A" in $A$-adequacy. 

Let $F$ be an orientable surface which has a finite (possibly empty) collection of points specified on $\partial F$ if $\partial F \not= \emptyset$. A link diagram on $F$ consists of finitely many arcs and closed curves on $F$ such that 
\begin{itemize}
\item There are finitely many transverse crossings with an over-strand and an under-strand. 
\item The endpoints of the arcs form a subset of the specified points on $\partial F$. In other words, the arcs are \emph{properly-embedded}. 
\end{itemize} 
Two link diagrams on $F$ are isotopic if they differ by a homeomorphism of $F$ isotopic to the identity. The isotopy is required to fix $\partial F$. 

\begin{defn}\label{defn:skein} Let $q$ be a fixed complex number. The \emph{linear skein} $\mathcal{S}(F)$ of $F$ is the vector space of formal linear sums over $\mathbb{C}$ of isotopy classes of link diagrams in $F$ quotiented by the relations 
\begin{enumerate}[(i)]
\item $D \cup \vcenter{\hbox{\includegraphics[scale=.10]{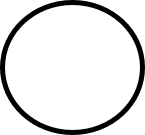}}} = (-q - q^{-1}) D$, and 
\item $ \vcenter{\hbox{\includegraphics[scale=.2]{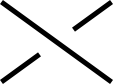}}} = q^{1/2} \ \vcenter{\hbox{\includegraphics[scale=.2]{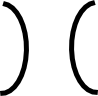}}} \ + q^{-1/2} \ \vcenter{\hbox{\includegraphics[scale=.2]{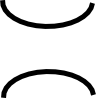}}} \ .$
\end{enumerate}  
\end{defn} 
We consider the linear skein $\mathcal{S}(D, n)$ of the disc $D$ with $2n$ points specified on its boundary, which will be represented as a square with $n$ marked points above and below. The square may be rotated with the marked points going from left to right. For $D_1, D_2 \in \Sk(D,n)$, there is a natural multiplication operation $D_1\cdot D_2$ defined by identifying the top boundary of $D_1$ with the bottom boundary of $D_2$.  This makes $\mathcal{S}(D, n)$ into an algebra $TL_n$, called the $n$th \emph{Temperley-Lieb algebra}. The algebra $TL_n$ is generated by crossing-less matchings $1_n, e^{1}_n, \ldots, e^{n-1}_{n}$. See Figure \ref{fig:TLgen} for an example. 
\begin{figure}[ht] 
\centering
\def\svgwidth{.7\columnwidth}
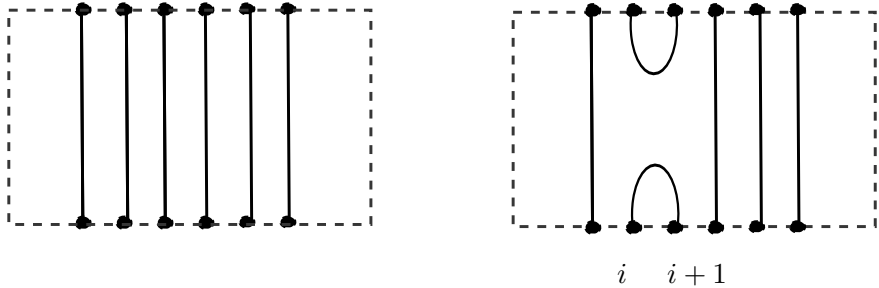
\caption{An example of the identity element $|_n$ and a generator $e^i_n$ of $TL_n$ for $n=6$ and $i=2$. \label{fig:TLgen}}
\end{figure}

We will use a shorthand notation which denotes $n$ parallel strands, the identity $1_n$, by $|_n$. 

Suppose that $q^2$ is not a $k$th root of unity for $k\leq n$. There is an element $\vcenter{\hbox{\def\svgwidth{.020\columnwidth} 
\begingroup%
  \makeatletter%
  \providecommand\color[2][]{%
    \errmessage{(Inkscape) Color is used for the text in Inkscape, but the package 'color.sty' is not loaded}%
    \renewcommand\color[2][]{}%
  }%
  \providecommand\transparent[1]{%
    \errmessage{(Inkscape) Transparency is used (non-zero) for the text in Inkscape, but the package 'transparent.sty' is not loaded}%
    \renewcommand\transparent[1]{}%
  }%
  \providecommand\rotatebox[2]{#2}%
  \ifx\svgwidth\undefined%
    \setlength{\unitlength}{69.35bp}%
    \ifx\svgscale\undefined%
      \relax%
    \else%
      \setlength{\unitlength}{\unitlength * \real{\svgscale}}%
    \fi%
  \else%
    \setlength{\unitlength}{\svgwidth}%
  \fi%
  \global\let\svgwidth\undefined%
  \global\let\svgscale\undefined%
  \makeatother%
  \begin{picture}(1,1.38428262)%
    \put(0,0){\includegraphics[width=\unitlength]{jwproj.pdf}}%
  \end{picture}%
\endgroup%
}}_n$ in $TL_n$ called the \emph{Jones-Wenzl projector}, which is uniquely defined  by the following properties. For the original reference where the projector was defined and studied, see \cite{Wen87}. 
\begin{enumerate}[(i)]
\item $\vcenter{\hbox{\def\svgwidth{.020\columnwidth} 
\begingroup%
  \makeatletter%
  \providecommand\color[2][]{%
    \errmessage{(Inkscape) Color is used for the text in Inkscape, but the package 'color.sty' is not loaded}%
    \renewcommand\color[2][]{}%
  }%
  \providecommand\transparent[1]{%
    \errmessage{(Inkscape) Transparency is used (non-zero) for the text in Inkscape, but the package 'transparent.sty' is not loaded}%
    \renewcommand\transparent[1]{}%
  }%
  \providecommand\rotatebox[2]{#2}%
  \ifx\svgwidth\undefined%
    \setlength{\unitlength}{69.35bp}%
    \ifx\svgscale\undefined%
      \relax%
    \else%
      \setlength{\unitlength}{\unitlength * \real{\svgscale}}%
    \fi%
  \else%
    \setlength{\unitlength}{\svgwidth}%
  \fi%
  \global\let\svgwidth\undefined%
  \global\let\svgscale\undefined%
  \makeatother%
  \begin{picture}(1,1.38428262)%
    \put(0,0){\includegraphics[width=\unitlength]{jwproj.pdf}}%
  \end{picture}%
\endgroup%
}}_n \cdot e^i_n = e^i_n \cdot \vcenter{\hbox{\def\svgwidth{.020\columnwidth} 
\begingroup%
  \makeatletter%
  \providecommand\color[2][]{%
    \errmessage{(Inkscape) Color is used for the text in Inkscape, but the package 'color.sty' is not loaded}%
    \renewcommand\color[2][]{}%
  }%
  \providecommand\transparent[1]{%
    \errmessage{(Inkscape) Transparency is used (non-zero) for the text in Inkscape, but the package 'transparent.sty' is not loaded}%
    \renewcommand\transparent[1]{}%
  }%
  \providecommand\rotatebox[2]{#2}%
  \ifx\svgwidth\undefined%
    \setlength{\unitlength}{69.35bp}%
    \ifx\svgscale\undefined%
      \relax%
    \else%
      \setlength{\unitlength}{\unitlength * \real{\svgscale}}%
    \fi%
  \else%
    \setlength{\unitlength}{\svgwidth}%
  \fi%
  \global\let\svgwidth\undefined%
  \global\let\svgscale\undefined%
  \makeatother%
  \begin{picture}(1,1.38428262)%
    \put(0,0){\includegraphics[width=\unitlength]{jwproj.pdf}}%
  \end{picture}%
\endgroup%
}} =0$ for $1 \leq i \leq n-1$. \label{list:prop1}
\item $\vcenter{\hbox{\def\svgwidth{.020\columnwidth} 
\begingroup%
  \makeatletter%
  \providecommand\color[2][]{%
    \errmessage{(Inkscape) Color is used for the text in Inkscape, but the package 'color.sty' is not loaded}%
    \renewcommand\color[2][]{}%
  }%
  \providecommand\transparent[1]{%
    \errmessage{(Inkscape) Transparency is used (non-zero) for the text in Inkscape, but the package 'transparent.sty' is not loaded}%
    \renewcommand\transparent[1]{}%
  }%
  \providecommand\rotatebox[2]{#2}%
  \ifx\svgwidth\undefined%
    \setlength{\unitlength}{69.35bp}%
    \ifx\svgscale\undefined%
      \relax%
    \else%
      \setlength{\unitlength}{\unitlength * \real{\svgscale}}%
    \fi%
  \else%
    \setlength{\unitlength}{\svgwidth}%
  \fi%
  \global\let\svgwidth\undefined%
  \global\let\svgscale\undefined%
  \makeatother%
  \begin{picture}(1,1.38428262)%
    \put(0,0){\includegraphics[width=\unitlength]{jwproj.pdf}}%
  \end{picture}%
\endgroup%
}}_n - |_n $ belongs to the algebra generated by $\{e^1_n, e^2_n,\ldots, e^{n-1}_n\}$. 
\item $\vcenter{\hbox{\def\svgwidth{.020\columnwidth} 
\begingroup%
  \makeatletter%
  \providecommand\color[2][]{%
    \errmessage{(Inkscape) Color is used for the text in Inkscape, but the package 'color.sty' is not loaded}%
    \renewcommand\color[2][]{}%
  }%
  \providecommand\transparent[1]{%
    \errmessage{(Inkscape) Transparency is used (non-zero) for the text in Inkscape, but the package 'transparent.sty' is not loaded}%
    \renewcommand\transparent[1]{}%
  }%
  \providecommand\rotatebox[2]{#2}%
  \ifx\svgwidth\undefined%
    \setlength{\unitlength}{69.35bp}%
    \ifx\svgscale\undefined%
      \relax%
    \else%
      \setlength{\unitlength}{\unitlength * \real{\svgscale}}%
    \fi%
  \else%
    \setlength{\unitlength}{\svgwidth}%
  \fi%
  \global\let\svgwidth\undefined%
  \global\let\svgscale\undefined%
  \makeatother%
  \begin{picture}(1,1.38428262)%
    \put(0,0){\includegraphics[width=\unitlength]{jwproj.pdf}}%
  \end{picture}%
\endgroup%
}}_n \cdot \vcenter{\hbox{\def\svgwidth{.020\columnwidth} 
\begingroup%
  \makeatletter%
  \providecommand\color[2][]{%
    \errmessage{(Inkscape) Color is used for the text in Inkscape, but the package 'color.sty' is not loaded}%
    \renewcommand\color[2][]{}%
  }%
  \providecommand\transparent[1]{%
    \errmessage{(Inkscape) Transparency is used (non-zero) for the text in Inkscape, but the package 'transparent.sty' is not loaded}%
    \renewcommand\transparent[1]{}%
  }%
  \providecommand\rotatebox[2]{#2}%
  \ifx\svgwidth\undefined%
    \setlength{\unitlength}{69.35bp}%
    \ifx\svgscale\undefined%
      \relax%
    \else%
      \setlength{\unitlength}{\unitlength * \real{\svgscale}}%
    \fi%
  \else%
    \setlength{\unitlength}{\svgwidth}%
  \fi%
  \global\let\svgwidth\undefined%
  \global\let\svgscale\undefined%
  \makeatother%
  \begin{picture}(1,1.38428262)%
    \put(0,0){\includegraphics[width=\unitlength]{jwproj.pdf}}%
  \end{picture}%
\endgroup%
}}_n = \vcenter{\hbox{\def\svgwidth{.020\columnwidth} 
\begingroup%
  \makeatletter%
  \providecommand\color[2][]{%
    \errmessage{(Inkscape) Color is used for the text in Inkscape, but the package 'color.sty' is not loaded}%
    \renewcommand\color[2][]{}%
  }%
  \providecommand\transparent[1]{%
    \errmessage{(Inkscape) Transparency is used (non-zero) for the text in Inkscape, but the package 'transparent.sty' is not loaded}%
    \renewcommand\transparent[1]{}%
  }%
  \providecommand\rotatebox[2]{#2}%
  \ifx\svgwidth\undefined%
    \setlength{\unitlength}{69.35bp}%
    \ifx\svgscale\undefined%
      \relax%
    \else%
      \setlength{\unitlength}{\unitlength * \real{\svgscale}}%
    \fi%
  \else%
    \setlength{\unitlength}{\svgwidth}%
  \fi%
  \global\let\svgwidth\undefined%
  \global\let\svgscale\undefined%
  \makeatother%
  \begin{picture}(1,1.38428262)%
    \put(0,0){\includegraphics[width=\unitlength]{jwproj.pdf}}%
  \end{picture}%
\endgroup%
}}_n$.
\item Let $\mathcal{S}(S^1 \times I)$ be the linear skein of  the annuli with no points marked on its boundaries. The image of \ $\vcenter{\hbox{\def\svgwidth{.020\columnwidth} 
\begingroup%
  \makeatletter%
  \providecommand\color[2][]{%
    \errmessage{(Inkscape) Color is used for the text in Inkscape, but the package 'color.sty' is not loaded}%
    \renewcommand\color[2][]{}%
  }%
  \providecommand\transparent[1]{%
    \errmessage{(Inkscape) Transparency is used (non-zero) for the text in Inkscape, but the package 'transparent.sty' is not loaded}%
    \renewcommand\transparent[1]{}%
  }%
  \providecommand\rotatebox[2]{#2}%
  \ifx\svgwidth\undefined%
    \setlength{\unitlength}{69.35bp}%
    \ifx\svgscale\undefined%
      \relax%
    \else%
      \setlength{\unitlength}{\unitlength * \real{\svgscale}}%
    \fi%
  \else%
    \setlength{\unitlength}{\svgwidth}%
  \fi%
  \global\let\svgwidth\undefined%
  \global\let\svgscale\undefined%
  \makeatother%
  \begin{picture}(1,1.38428262)%
    \put(0,0){\includegraphics[width=\unitlength]{jwproj.pdf}}%
  \end{picture}%
\endgroup%
}}_n$ in $\mathcal{S}(\mathbb{R}^2)$ obtained from first sending it to $\Sk(S^1\times I)$ by joining the $n$ boundary points on the top with those at the bottom, and then embedding $S^1\times I$ in $\mathbb{R}^2$, is equal to 
\[ \vcenter{\hbox{\includegraphics[scale=.2]{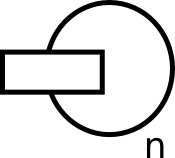}}} = (-1)^n[n] \cdot \langle \text{the empty diagram on $\mathbb{R}^2$} \rangle, \] \label{list:prop4}
\end{enumerate}
where $[n]$ is the \emph{quantum integer} defined by
\[ [n]:= \frac{q^{-(n+1)} - q^{n+1}}{q^{-1}-q}. \] 

From the defining properties, the Jones-Wenzl idempotent also satisfies a recursion relation \eqref{eq:jwrecursive} and another identity \eqref{eq:jwid} as indicated in Figure \ref{fig:jwrecursive} and Figure \ref{fig:jwid}. 
\begin{figure}[ht]
\centering
\begin{equation} \label{eq:jwrecursive}
\def\svgwidth{.7\columnwidth}
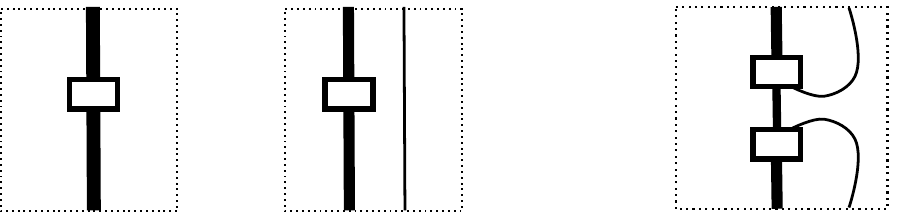
\end{equation}
\caption{\label{fig:jwrecursive}A recursive relation for the Jones-Wenzl idempotent.}
\end{figure} 
\begin{figure}[ht]
\begin{equation} \label{eq:jwid}
\def\svgwidth{.4\columnwidth}
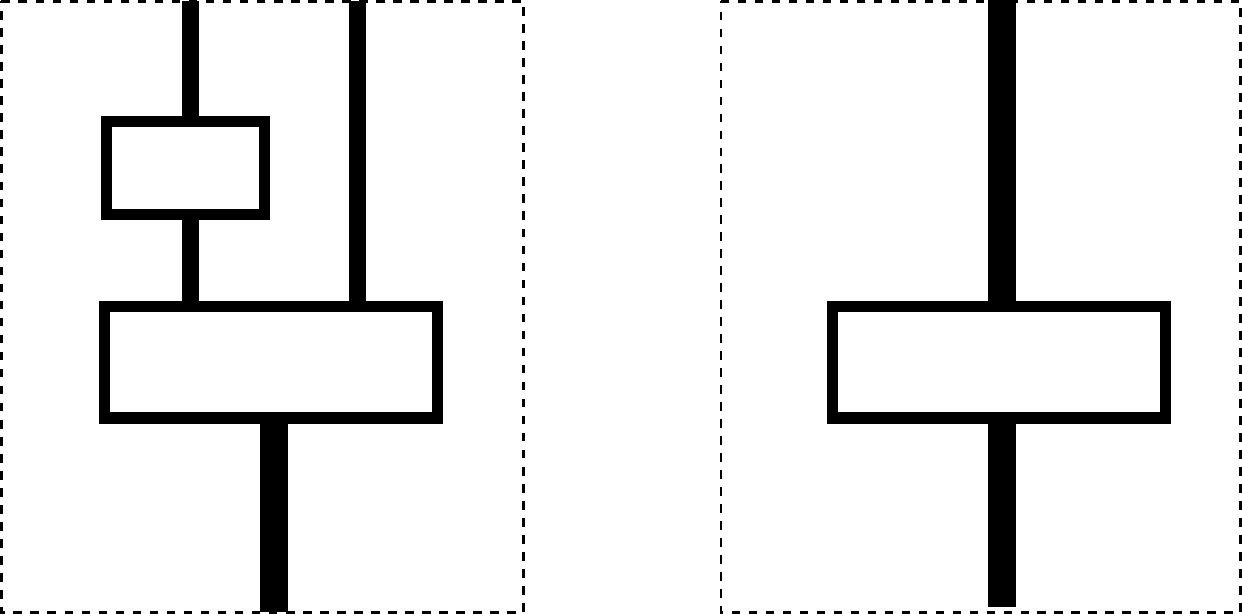
\end{equation}
\caption{\label{fig:jwid} An identity for the Jones-Wenzl idempotent.}
\end{figure} 

\subsection{The colored Jones polynomial}

Here we give a definition of the colored Jones polynomial in terms of skeins in the Temperley-Lieb algebra defined in Section \ref{subsec:skein}. For the original construction of the invariant based on the representation theory of $U_q(\mathfrak{sl}_2)$ for generic $q$, see \cite{RT90}. A discussion and proof of the equivalence of two definitions may be found in \cite{Cos14}. 

\begin{defn}\label{defn:cp}
Let $D$ be a diagram of a link $K\subset S^3$ with $k$ components. For each component $D_i$ for $i \in \{1,\ldots, k\}$ of $D$ take an annulus $A_i$ via the blackboard framing.
 Let 
\[ f_D: \underbrace{\Sk(S^1\times I) \times \cdots \times \Sk(S^1 \times I)}_{k \text{ times }} \rightarrow \Sk(\mathbb{R}^2),     \] be the map which sends a $k$-tuple of elements $(s_1, \ldots, s_k)$ to $S(\mathbb{R}^2)$ by immersing the collection of annuli containing the skeins in the plane such that the over- and under-crossings of $D$ are the over- and under-crossings of the annuli. 
 The \emph{$n$th unreduced colored Jones polynomial} $J_K(q; n)$ may be defined as 
\[ J_K(q; n) := ((-1)^nq^{\frac{n^2+2n}{2}})^{\omega(D)} \left\langle f_D\underbrace{\left(\vcenter{\hbox{\includegraphics[scale=.15]{jwprojc.png}}},  \vcenter{\hbox{\includegraphics[scale=.15]{jwprojc.png}}}, \cdots, \vcenter{\hbox{\includegraphics[scale=.15]{jwprojc.png}}} \right)}_{k \text{ times }} \right\rangle. \] 

See the right-hand side of Figure \ref{fig:tcompose} for an example where $D$ has $2$ components.
The Kauffman bracket here is extended by linearity and gives the polynomial multiplying the empty diagram after reducing the diagram via skein relations. Note that this gives $J_{\vcenter{\hbox{\includegraphics[scale=.05]{circ.png}}}}(q; 1) = (-1)^n[n]$ as the normalization. To simplify notation we will denote the skein $f_D \left(\vcenter{\hbox{\includegraphics[scale=.15]{jwprojc.png}}},  \vcenter{\hbox{\includegraphics[scale=.15]{jwprojc.png}}}, \cdots, \vcenter{\hbox{\includegraphics[scale=.15]{jwprojc.png}}} \right)$ from now on by $D^n_{\vcenter{\hbox{\def\svgwidth{.020\columnwidth} 
\begingroup%
  \makeatletter%
  \providecommand\color[2][]{%
    \errmessage{(Inkscape) Color is used for the text in Inkscape, but the package 'color.sty' is not loaded}%
    \renewcommand\color[2][]{}%
  }%
  \providecommand\transparent[1]{%
    \errmessage{(Inkscape) Transparency is used (non-zero) for the text in Inkscape, but the package 'transparent.sty' is not loaded}%
    \renewcommand\transparent[1]{}%
  }%
  \providecommand\rotatebox[2]{#2}%
  \ifx\svgwidth\undefined%
    \setlength{\unitlength}{69.35bp}%
    \ifx\svgscale\undefined%
      \relax%
    \else%
      \setlength{\unitlength}{\unitlength * \real{\svgscale}}%
    \fi%
  \else%
    \setlength{\unitlength}{\svgwidth}%
  \fi%
  \global\let\svgwidth\undefined%
  \global\let\svgscale\undefined%
  \makeatother%
  \begin{picture}(1,1.38428262)%
    \put(0,0){\includegraphics[width=\unitlength]{jwproj.pdf}}%
  \end{picture}%
\endgroup%
}}}$. 
\end{defn} 

We now describe the categorification of the colored Jones polynomial, which is a homology theory whose Euler characteristic is equal to the polynomial. Here we summarize the description of \cite{Bar05} and \cite{Roz12}. The differences in conventions between those used in this paper and those authors are explained at the beginning of each section.

\subsection{Khovanov bracket for tangles}
We follow the approach of \cite{Bar05} with a different grading convention: Our graded category has grading in half-integers, and  we do not consider complexes with bi-grading shifted by the number of positive and negative crossings of the tangle. This results in a chain complex that is invariant up to a shift under a framing change which is specified in Section \ref{subsec:catcjp}. For details, please consult \cite{Bar05}.

We work in the additive category $\mathcal{TL}$ constructed from the category $TL$ whose objects are finite sets of properly-embedded arcs with no crossings and simple closed curves in $\Sk(D, 2n)$ for $n\in \mathbb{Z}_+$, considered up to boundary-preserving isotopy. For two such skeins in $TL$, a morphism between them is a cobordism which is a compact 2-manifold whose boundary is the disjoint union of the two skeins. The composition of morphisms is given by placing one cobordism on top of another in the obvious way. From $TL$ see \cite{Bar05} for the details for the construction of $\mathcal{TL}$, where he shows that one can construct an additive category from any category. The process formally adds finite direct sums of $\mathbb{Z}$-linear combinations of objects and morphisms to $TL$ in such a way that the composition of morphisms is bilinear. 

From $\mathcal{TL}$ we consider the homotopy category of chain complexes over $\mathcal{TL}$ which are bounded from below denoted by $Kom^+(\mathcal{TL})$. Let $m \in \mathbb{Z}$ and $T_i$ be an object of $\mathcal{TL}$. An object of $Kom^+(\mathcal{TL})$ is a chain complex $(\textbf{T}, d_{\textbf{T}})$
\begin{equation} \label{eq:chain} \textbf{T} = \cdots \rightarrow T_{i+1} \stackrel{d_{i+1}}{\rightarrow} T_{i} \stackrel{d_{i}}{\rightarrow} \cdots  \rightarrow T_{m}, \end{equation}
considered up to homotopy equivalence $\sim$, and a morphism between two chain complexes is a chain map. We use the notation $\textbf{T}_i := T_i$. We further mod out the set of morphisms $Kom^+(\mathcal{TL})$ by the $S, T$, and $4Tu$-relations, and consider as in \cite{Bar05}, the \emph{graded} category $Kob(\mathcal{TL})$ where grading shifts of complexes are induced by cobordisms. 

To an $(n, n)$-tangle $T$ with crossings in $\Sk(D, n)$ is then associated a chain complex $(\textbf{T}, d_{\textbf{T}})$, where $T_{i}$'s are smoothings of $T$ and the boundary map $d_{\textbf{T}}$ comes from the saddle cobordism 
\[ \vcenter{\hbox{\includegraphics[scale=.2]{crossing2.png}}} \ \stackrel{s}{\rightarrow} \ \vcenter{\hbox{\includegraphics[scale=.2]{crossing3.png}}} \] from the $B$-resolution of the crossing $\vcenter{\hbox{\includegraphics[scale=.2]{crossing1.png}}}$ to the $A$-resolution. For a tangle $T\in TL_n$ we will often represent the associated complex $\mathbf{T}$ in $Kob(\mathcal{TL})$ by a picture of the tangle itself. We may compose the complexes associated to tangles in the fashion of planar algebras as also detailed in \cite{Bar05}. After the application of a TQFT this gives the Khovanov homology of tangles. We specify our grading conventions in Section \ref{subsec:notkbracket}.

\subsection{Notation and grading for the Khovanov bracket} \label{subsec:notkbracket}

Our grading conventions are adapted from \cite{Roz12} with a minor change of variables. The difference with Rozansky's convention in \cite{Roz12} is that his $\textbf{q}$ there is our $\textbf{q}^{-1}$ and his $\textbf{h}$ there is our $\textbf{h}^{-1}$. 

Let $\textbf{A}$ be a bi-graded chain complex $(A_{i,j}, d_A)$, and let $\textbf{A}[1]$ be the chain complex whose homological grading $i$ is shifted by 1 with differential $-d_A$, so $\textbf{A}[1]_i = A_{i-1}$, while $\textbf{A}\{1\}$ has the $j$-grading shifted by 1. Following Rozansky's convention in using a non-standard notation for the grading-shift, we will also use the notation $\mathbf{h}\mathbf{A}$ for $\textbf{A}[1]$, with 
\[\mathbf{h}^k\mathbf{A}=\mathbf{A}[k], \] and $\mathbf{q}\mathbf{A}$ for $\textbf{A}\{1\}$.
This will completely specify the changes to the bi-grading in what follows. 
Our grading convention is completely determined by the following presentation of the Khovanov bracket. Here $\mathbf{h}^{\pm \frac{1}{2}}$ simply shifts the $i$-degree of the complex by $\pm \frac{1}{2}$. 
\begin{enumerate}[(i)]
\item $\vcenter{\hbox{\includegraphics[scale=.15]{circ.png}}} = \textbf{q} \ \vcenter{\hbox{\includegraphics[scale=.2]{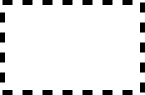}}}  + \textbf{q}^{-1} \ \vcenter{\hbox{\includegraphics[scale=.2]{blank.png}}}$ \ , where $\vcenter{\hbox{\includegraphics[scale=.2]{blank.png}}}$ is the empty skein, and
\item \label{lem:multicone} $ \vcenter{\hbox{\includegraphics[scale=.2]{crossing1.png}}} = \text{Cone}\left(\textbf{h}^{-\frac{1}{2}} \ \vcenter{\hbox{\includegraphics[scale=.2]{crossing2.png}}} \ \stackrel{s}{\rightarrow} \textbf{h}^{-\frac{1}{2}} \vcenter{\hbox{\includegraphics[scale=.2]{crossing3.png}}}\right)\ = \fbox{$\textbf{h}^{\frac{1}{2}} \ \vcenter{\hbox{\includegraphics[scale=.2]{crossing2.png}}} \rightarrow \textbf{h}^{-\frac{1}{2}}\vcenter{\hbox{\includegraphics[scale=.2]{crossing3.png}}} $} ,$ where $s$ is the saddle cobordism. Note that $\deg_{\mathbf{h}}(s)=-1$. 
\end{enumerate}

The notation $\text{Cone}\left(\mathbf{A} \stackrel{f}{\rightarrow} \mathbf{B} \right)$ for two chain complexes $(\mathbf{A}, d_{\mathbf{A}})$ and $(\mathbf{B}, d_{\mathbf{B}})$ indicates the complex of the mapping cone $\mathbf{h}\mathbf{A}\oplus \mathbf{B}$ with differential $d_f = \left[ \begin{array}{cc} -d_\textbf{A} & 0 \\ f & d_{\textbf{B}}  \end{array}\right]$. We will also use the following notation for the mapping cone. The map $f$ will be suppressed whenever it is clear. 

\begin{figure}[ht]
\def\svgwidth{.5\columnwidth}
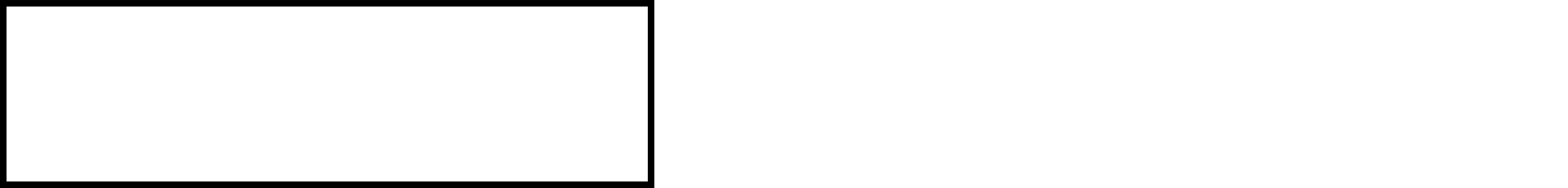
\end{figure}

\subsection{The chain complex for the Jones-Wenzl projector}

Rozansky constructs a chain complex in $Kob(\mathcal{TL})$, denoted here by $\vcenter{\hbox{\def\svgwidth{.020\columnwidth} 
\begingroup%
  \makeatletter%
  \providecommand\color[2][]{%
    \errmessage{(Inkscape) Color is used for the text in Inkscape, but the package 'color.sty' is not loaded}%
    \renewcommand\color[2][]{}%
  }%
  \providecommand\transparent[1]{%
    \errmessage{(Inkscape) Transparency is used (non-zero) for the text in Inkscape, but the package 'transparent.sty' is not loaded}%
    \renewcommand\transparent[1]{}%
  }%
  \providecommand\rotatebox[2]{#2}%
  \ifx\svgwidth\undefined%
    \setlength{\unitlength}{69.35bp}%
    \ifx\svgscale\undefined%
      \relax%
    \else%
      \setlength{\unitlength}{\unitlength * \real{\svgscale}}%
    \fi%
  \else%
    \setlength{\unitlength}{\svgwidth}%
  \fi%
  \global\let\svgwidth\undefined%
  \global\let\svgscale\undefined%
  \makeatother%
  \begin{picture}(1,1.38428262)%
    \put(0,0){\includegraphics[width=\unitlength]{jwproj.pdf}}%
  \end{picture}%
\endgroup%
}}_n$, which satisfies the following universal properties.

\begin{defn} \label{defn:mhd} Given a chain complex $\mathbf{T}:\cdots \rightarrow T_{i+1} \rightarrow T_{i} \rightarrow \cdots \rightarrow T_{m} $ in $Kob(\mathcal{TL})$, the \emph{minimum homological degree} of $\mathbf{T}$, denoted by $|\mathbf{T}|$, is $m$.  

\begin{thm}[{\cite[Theorem 2.7]{Roz10}}] \label{thm:uniprop}
Let $e^i_n$ be a generator of $TL_n$. 
The complex $\vcenter{\hbox{\def\svgwidth{.020\columnwidth} 
\begingroup%
  \makeatletter%
  \providecommand\color[2][]{%
    \errmessage{(Inkscape) Color is used for the text in Inkscape, but the package 'color.sty' is not loaded}%
    \renewcommand\color[2][]{}%
  }%
  \providecommand\transparent[1]{%
    \errmessage{(Inkscape) Transparency is used (non-zero) for the text in Inkscape, but the package 'transparent.sty' is not loaded}%
    \renewcommand\transparent[1]{}%
  }%
  \providecommand\rotatebox[2]{#2}%
  \ifx\svgwidth\undefined%
    \setlength{\unitlength}{69.35bp}%
    \ifx\svgscale\undefined%
      \relax%
    \else%
      \setlength{\unitlength}{\unitlength * \real{\svgscale}}%
    \fi%
  \else%
    \setlength{\unitlength}{\svgwidth}%
  \fi%
  \global\let\svgwidth\undefined%
  \global\let\svgscale\undefined%
  \makeatother%
  \begin{picture}(1,1.38428262)%
    \put(0,0){\includegraphics[width=\unitlength]{jwproj.pdf}}%
  \end{picture}%
\endgroup%
}}_n$ has the following properties: 
\begin{enumerate}[(i)]
\item The complex of a composition of $\vcenter{\hbox{\def\svgwidth{.020\columnwidth} 
\begingroup%
  \makeatletter%
  \providecommand\color[2][]{%
    \errmessage{(Inkscape) Color is used for the text in Inkscape, but the package 'color.sty' is not loaded}%
    \renewcommand\color[2][]{}%
  }%
  \providecommand\transparent[1]{%
    \errmessage{(Inkscape) Transparency is used (non-zero) for the text in Inkscape, but the package 'transparent.sty' is not loaded}%
    \renewcommand\transparent[1]{}%
  }%
  \providecommand\rotatebox[2]{#2}%
  \ifx\svgwidth\undefined%
    \setlength{\unitlength}{69.35bp}%
    \ifx\svgscale\undefined%
      \relax%
    \else%
      \setlength{\unitlength}{\unitlength * \real{\svgscale}}%
    \fi%
  \else%
    \setlength{\unitlength}{\svgwidth}%
  \fi%
  \global\let\svgwidth\undefined%
  \global\let\svgscale\undefined%
  \makeatother%
  \begin{picture}(1,1.38428262)%
    \put(0,0){\includegraphics[width=\unitlength]{jwproj.pdf}}%
  \end{picture}%
\endgroup%
}}_n$ with $e^i_n$ for all $i = 1,\ldots, n-1$ is contractible, i.e., \label{thmi:contractible}
\[ e^i_n \circ \vcenter{\hbox{\def\svgwidth{.020\columnwidth} 
\begingroup%
  \makeatletter%
  \providecommand\color[2][]{%
    \errmessage{(Inkscape) Color is used for the text in Inkscape, but the package 'color.sty' is not loaded}%
    \renewcommand\color[2][]{}%
  }%
  \providecommand\transparent[1]{%
    \errmessage{(Inkscape) Transparency is used (non-zero) for the text in Inkscape, but the package 'transparent.sty' is not loaded}%
    \renewcommand\transparent[1]{}%
  }%
  \providecommand\rotatebox[2]{#2}%
  \ifx\svgwidth\undefined%
    \setlength{\unitlength}{69.35bp}%
    \ifx\svgscale\undefined%
      \relax%
    \else%
      \setlength{\unitlength}{\unitlength * \real{\svgscale}}%
    \fi%
  \else%
    \setlength{\unitlength}{\svgwidth}%
  \fi%
  \global\let\svgwidth\undefined%
  \global\let\svgscale\undefined%
  \makeatother%
  \begin{picture}(1,1.38428262)%
    \put(0,0){\includegraphics[width=\unitlength]{jwproj.pdf}}%
  \end{picture}%
\endgroup%
}}_n \sim \vcenter{\hbox{\def\svgwidth{.020\columnwidth} 
\begingroup%
  \makeatletter%
  \providecommand\color[2][]{%
    \errmessage{(Inkscape) Color is used for the text in Inkscape, but the package 'color.sty' is not loaded}%
    \renewcommand\color[2][]{}%
  }%
  \providecommand\transparent[1]{%
    \errmessage{(Inkscape) Transparency is used (non-zero) for the text in Inkscape, but the package 'transparent.sty' is not loaded}%
    \renewcommand\transparent[1]{}%
  }%
  \providecommand\rotatebox[2]{#2}%
  \ifx\svgwidth\undefined%
    \setlength{\unitlength}{69.35bp}%
    \ifx\svgscale\undefined%
      \relax%
    \else%
      \setlength{\unitlength}{\unitlength * \real{\svgscale}}%
    \fi%
  \else%
    \setlength{\unitlength}{\svgwidth}%
  \fi%
  \global\let\svgwidth\undefined%
  \global\let\svgscale\undefined%
  \makeatother%
  \begin{picture}(1,1.38428262)%
    \put(0,0){\includegraphics[width=\unitlength]{jwproj.pdf}}%
  \end{picture}%
\endgroup%
}}_n \circ e^i_n \sim 0. \] 
\item The complex $\vcenter{\hbox{\def\svgwidth{.020\columnwidth} 
\begingroup%
  \makeatletter%
  \providecommand\color[2][]{%
    \errmessage{(Inkscape) Color is used for the text in Inkscape, but the package 'color.sty' is not loaded}%
    \renewcommand\color[2][]{}%
  }%
  \providecommand\transparent[1]{%
    \errmessage{(Inkscape) Transparency is used (non-zero) for the text in Inkscape, but the package 'transparent.sty' is not loaded}%
    \renewcommand\transparent[1]{}%
  }%
  \providecommand\rotatebox[2]{#2}%
  \ifx\svgwidth\undefined%
    \setlength{\unitlength}{69.35bp}%
    \ifx\svgscale\undefined%
      \relax%
    \else%
      \setlength{\unitlength}{\unitlength * \real{\svgscale}}%
    \fi%
  \else%
    \setlength{\unitlength}{\svgwidth}%
  \fi%
  \global\let\svgwidth\undefined%
  \global\let\svgscale\undefined%
  \makeatother%
  \begin{picture}(1,1.38428262)%
    \put(0,0){\includegraphics[width=\unitlength]{jwproj.pdf}}%
  \end{picture}%
\endgroup%
}}_n$ is idempotent with respect to tangle composition, i.e., 
\[ \vcenter{\hbox{\def\svgwidth{.020\columnwidth} 
\begingroup%
  \makeatletter%
  \providecommand\color[2][]{%
    \errmessage{(Inkscape) Color is used for the text in Inkscape, but the package 'color.sty' is not loaded}%
    \renewcommand\color[2][]{}%
  }%
  \providecommand\transparent[1]{%
    \errmessage{(Inkscape) Transparency is used (non-zero) for the text in Inkscape, but the package 'transparent.sty' is not loaded}%
    \renewcommand\transparent[1]{}%
  }%
  \providecommand\rotatebox[2]{#2}%
  \ifx\svgwidth\undefined%
    \setlength{\unitlength}{69.35bp}%
    \ifx\svgscale\undefined%
      \relax%
    \else%
      \setlength{\unitlength}{\unitlength * \real{\svgscale}}%
    \fi%
  \else%
    \setlength{\unitlength}{\svgwidth}%
  \fi%
  \global\let\svgwidth\undefined%
  \global\let\svgscale\undefined%
  \makeatother%
  \begin{picture}(1,1.38428262)%
    \put(0,0){\includegraphics[width=\unitlength]{jwproj.pdf}}%
  \end{picture}%
\endgroup%
}}_n\circ \vcenter{\hbox{\def\svgwidth{.020\columnwidth} 
\begingroup%
  \makeatletter%
  \providecommand\color[2][]{%
    \errmessage{(Inkscape) Color is used for the text in Inkscape, but the package 'color.sty' is not loaded}%
    \renewcommand\color[2][]{}%
  }%
  \providecommand\transparent[1]{%
    \errmessage{(Inkscape) Transparency is used (non-zero) for the text in Inkscape, but the package 'transparent.sty' is not loaded}%
    \renewcommand\transparent[1]{}%
  }%
  \providecommand\rotatebox[2]{#2}%
  \ifx\svgwidth\undefined%
    \setlength{\unitlength}{69.35bp}%
    \ifx\svgscale\undefined%
      \relax%
    \else%
      \setlength{\unitlength}{\unitlength * \real{\svgscale}}%
    \fi%
  \else%
    \setlength{\unitlength}{\svgwidth}%
  \fi%
  \global\let\svgwidth\undefined%
  \global\let\svgscale\undefined%
  \makeatother%
  \begin{picture}(1,1.38428262)%
    \put(0,0){\includegraphics[width=\unitlength]{jwproj.pdf}}%
  \end{picture}%
\endgroup%
}}_n \sim \vcenter{\hbox{\def\svgwidth{.020\columnwidth} 
\begingroup%
  \makeatletter%
  \providecommand\color[2][]{%
    \errmessage{(Inkscape) Color is used for the text in Inkscape, but the package 'color.sty' is not loaded}%
    \renewcommand\color[2][]{}%
  }%
  \providecommand\transparent[1]{%
    \errmessage{(Inkscape) Transparency is used (non-zero) for the text in Inkscape, but the package 'transparent.sty' is not loaded}%
    \renewcommand\transparent[1]{}%
  }%
  \providecommand\rotatebox[2]{#2}%
  \ifx\svgwidth\undefined%
    \setlength{\unitlength}{69.35bp}%
    \ifx\svgscale\undefined%
      \relax%
    \else%
      \setlength{\unitlength}{\unitlength * \real{\svgscale}}%
    \fi%
  \else%
    \setlength{\unitlength}{\svgwidth}%
  \fi%
  \global\let\svgwidth\undefined%
  \global\let\svgscale\undefined%
  \makeatother%
  \begin{picture}(1,1.38428262)%
    \put(0,0){\includegraphics[width=\unitlength]{jwproj.pdf}}%
  \end{picture}%
\endgroup%
}}_n. \] 
\item Let $\gamma \in TL_n$ be an $(n, n)$-tangle in the disk $D$ with $n$ marked points above and below. A through strand is an arc joining a point on the upper edge to a point on the lower edge. Let $|\gamma|$ denote the number of through strands of $\gamma$.

The complex $\vcenter{\hbox{\def\svgwidth{.020\columnwidth} 
\begingroup%
  \makeatletter%
  \providecommand\color[2][]{%
    \errmessage{(Inkscape) Color is used for the text in Inkscape, but the package 'color.sty' is not loaded}%
    \renewcommand\color[2][]{}%
  }%
  \providecommand\transparent[1]{%
    \errmessage{(Inkscape) Transparency is used (non-zero) for the text in Inkscape, but the package 'transparent.sty' is not loaded}%
    \renewcommand\transparent[1]{}%
  }%
  \providecommand\rotatebox[2]{#2}%
  \ifx\svgwidth\undefined%
    \setlength{\unitlength}{69.35bp}%
    \ifx\svgscale\undefined%
      \relax%
    \else%
      \setlength{\unitlength}{\unitlength * \real{\svgscale}}%
    \fi%
  \else%
    \setlength{\unitlength}{\svgwidth}%
  \fi%
  \global\let\svgwidth\undefined%
  \global\let\svgscale\undefined%
  \makeatother%
  \begin{picture}(1,1.38428262)%
    \put(0,0){\includegraphics[width=\unitlength]{jwproj.pdf}}%
  \end{picture}%
\endgroup%
}}_n$ has a presentation 
\[ \vcenter{\hbox{\def\svgwidth{.020\columnwidth} 
\begingroup%
  \makeatletter%
  \providecommand\color[2][]{%
    \errmessage{(Inkscape) Color is used for the text in Inkscape, but the package 'color.sty' is not loaded}%
    \renewcommand\color[2][]{}%
  }%
  \providecommand\transparent[1]{%
    \errmessage{(Inkscape) Transparency is used (non-zero) for the text in Inkscape, but the package 'transparent.sty' is not loaded}%
    \renewcommand\transparent[1]{}%
  }%
  \providecommand\rotatebox[2]{#2}%
  \ifx\svgwidth\undefined%
    \setlength{\unitlength}{69.35bp}%
    \ifx\svgscale\undefined%
      \relax%
    \else%
      \setlength{\unitlength}{\unitlength * \real{\svgscale}}%
    \fi%
  \else%
    \setlength{\unitlength}{\svgwidth}%
  \fi%
  \global\let\svgwidth\undefined%
  \global\let\svgscale\undefined%
  \makeatother%
  \begin{picture}(1,1.38428262)%
    \put(0,0){\includegraphics[width=\unitlength]{jwproj.pdf}}%
  \end{picture}%
\endgroup%
}}_n = \text{Cone}(\vcenter{\hbox{\def\svgwidth{.020\columnwidth} 
\begingroup%
  \makeatletter%
  \providecommand\color[2][]{%
    \errmessage{(Inkscape) Color is used for the text in Inkscape, but the package 'color.sty' is not loaded}%
    \renewcommand\color[2][]{}%
  }%
  \providecommand\transparent[1]{%
    \errmessage{(Inkscape) Transparency is used (non-zero) for the text in Inkscape, but the package 'transparent.sty' is not loaded}%
    \renewcommand\transparent[1]{}%
  }%
  \providecommand\rotatebox[2]{#2}%
  \ifx\svgwidth\undefined%
    \setlength{\unitlength}{70.15765381bp}%
    \ifx\svgscale\undefined%
      \relax%
    \else%
      \setlength{\unitlength}{\unitlength * \real{\svgscale}}%
    \fi%
  \else%
    \setlength{\unitlength}{\svgwidth}%
  \fi%
  \global\let\svgwidth\undefined%
  \global\let\svgscale\undefined%
  \makeatother%
  \begin{picture}(1,1.44816701)%
    \put(0,0){\includegraphics[width=\unitlength]{jwprojp.pdf}}%
  \end{picture}%
\endgroup%
}} \stackrel{f}{\rightarrow} |_n),\] 
where 
\begin{equation} \vcenter{\hbox{\def\svgwidth{.020\columnwidth} 
\begingroup%
  \makeatletter%
  \providecommand\color[2][]{%
    \errmessage{(Inkscape) Color is used for the text in Inkscape, but the package 'color.sty' is not loaded}%
    \renewcommand\color[2][]{}%
  }%
  \providecommand\transparent[1]{%
    \errmessage{(Inkscape) Transparency is used (non-zero) for the text in Inkscape, but the package 'transparent.sty' is not loaded}%
    \renewcommand\transparent[1]{}%
  }%
  \providecommand\rotatebox[2]{#2}%
  \ifx\svgwidth\undefined%
    \setlength{\unitlength}{70.15765381bp}%
    \ifx\svgscale\undefined%
      \relax%
    \else%
      \setlength{\unitlength}{\unitlength * \real{\svgscale}}%
    \fi%
  \else%
    \setlength{\unitlength}{\svgwidth}%
  \fi%
  \global\let\svgwidth\undefined%
  \global\let\svgscale\undefined%
  \makeatother%
  \begin{picture}(1,1.44816701)%
    \put(0,0){\includegraphics[width=\unitlength]{jwprojp.pdf}}%
  \end{picture}%
\endgroup%
}}_n = \cdots \rightarrow \textbf{h}^i\bigoplus_{0\leq j \leq i, \gamma \in TL_n, |\gamma| <n } \mu_{ij, \gamma} \textbf{q}^j \ \gamma \rightarrow \cdots,  \end{equation}
and where $\gamma\in TL_n$ has $|\gamma | <n$ and $|\vcenter{\hbox{\def\svgwidth{.020\columnwidth} 
\begingroup%
  \makeatletter%
  \providecommand\color[2][]{%
    \errmessage{(Inkscape) Color is used for the text in Inkscape, but the package 'color.sty' is not loaded}%
    \renewcommand\color[2][]{}%
  }%
  \providecommand\transparent[1]{%
    \errmessage{(Inkscape) Transparency is used (non-zero) for the text in Inkscape, but the package 'transparent.sty' is not loaded}%
    \renewcommand\transparent[1]{}%
  }%
  \providecommand\rotatebox[2]{#2}%
  \ifx\svgwidth\undefined%
    \setlength{\unitlength}{69.35bp}%
    \ifx\svgscale\undefined%
      \relax%
    \else%
      \setlength{\unitlength}{\unitlength * \real{\svgscale}}%
    \fi%
  \else%
    \setlength{\unitlength}{\svgwidth}%
  \fi%
  \global\let\svgwidth\undefined%
  \global\let\svgscale\undefined%
  \makeatother%
  \begin{picture}(1,1.38428262)%
    \put(0,0){\includegraphics[width=\unitlength]{jwproj.pdf}}%
  \end{picture}%
\endgroup%
}}_n| = 0$. 
\end{enumerate}
\end{thm} 

 The reader may refer to \cite{Roz10} for the details of constructing  $\vcenter{\hbox{\def\svgwidth{.020\columnwidth} 
\begingroup%
  \makeatletter%
  \providecommand\color[2][]{%
    \errmessage{(Inkscape) Color is used for the text in Inkscape, but the package 'color.sty' is not loaded}%
    \renewcommand\color[2][]{}%
  }%
  \providecommand\transparent[1]{%
    \errmessage{(Inkscape) Transparency is used (non-zero) for the text in Inkscape, but the package 'transparent.sty' is not loaded}%
    \renewcommand\transparent[1]{}%
  }%
  \providecommand\rotatebox[2]{#2}%
  \ifx\svgwidth\undefined%
    \setlength{\unitlength}{69.35bp}%
    \ifx\svgscale\undefined%
      \relax%
    \else%
      \setlength{\unitlength}{\unitlength * \real{\svgscale}}%
    \fi%
  \else%
    \setlength{\unitlength}{\svgwidth}%
  \fi%
  \global\let\svgwidth\undefined%
  \global\let\svgscale\undefined%
  \makeatother%
  \begin{picture}(1,1.38428262)%
    \put(0,0){\includegraphics[width=\unitlength]{jwproj.pdf}}%
  \end{picture}%
\endgroup%
}}_n$ via the chain complexes for torus braids. We shall not use his construction in detail. We need only that the complex for the projector satisfies the universal properties listed in Theorem \ref{thm:uniprop} . There are two mutually dual categorifications from this approach. For this paper we are using the complex whose $i$-grading is bounded from below.

We will also make use of the categorified versions of \eqref{eq:jwrecursive}, \eqref{eq:jwid} proven in \cite{Roz12}, which replaces tangles by their categorification complexes and equality by homotopy equivalence. 

\begin{lem}{\cite[Proposition 3.6]{Roz12}}  \label{lem:jwidc12}

\begin{figure}[H]
\centering
\def\svgwidth{.5\columnwidth}
    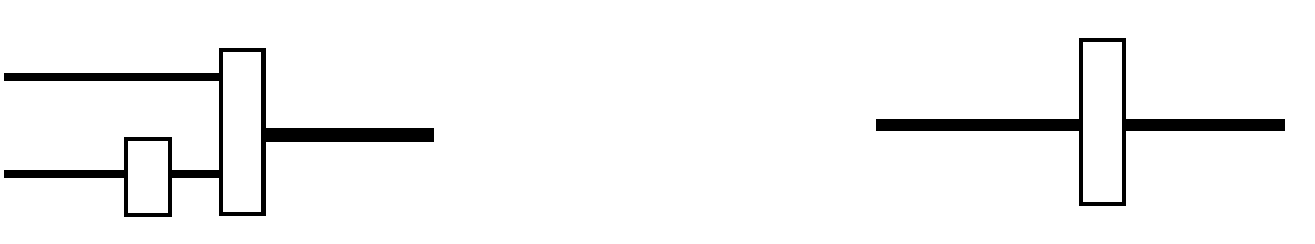
\end{figure} 
\end{lem} 

\begin{lem}{\cite[Theorem 3.8]{Roz12}} \label{lem:jwmcone} The $n+1$ strand categorified Jones-Wenzl projector has the following mapping cone presentation.
\begin{figure}[H]
    \centering
    \def\svgwidth{.8\columnwidth}
    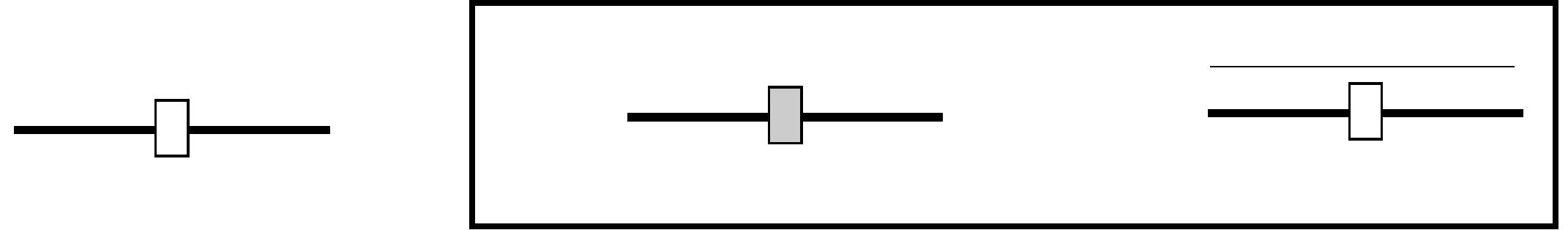
\end{figure}

\end{lem}

\subsection{Link homology, Euler characteristic, and the colored Jones polynomial} \label{subsec:catcjp}

The chain complex in $Kob(\mathcal{TL})$, to which a 1+1 TQFT is applied to obtain the homology groups categorifying the $n$th colored Jones polynomial, is obtained by composing the complex of the tangle $T^n$ from the $n$th cabled link component, with the complex of the $n$th Jones-Wenzl projector $\vcenter{\hbox{\def\svgwidth{.020\columnwidth} 
\begingroup%
  \makeatletter%
  \providecommand\color[2][]{%
    \errmessage{(Inkscape) Color is used for the text in Inkscape, but the package 'color.sty' is not loaded}%
    \renewcommand\color[2][]{}%
  }%
  \providecommand\transparent[1]{%
    \errmessage{(Inkscape) Transparency is used (non-zero) for the text in Inkscape, but the package 'transparent.sty' is not loaded}%
    \renewcommand\transparent[1]{}%
  }%
  \providecommand\rotatebox[2]{#2}%
  \ifx\svgwidth\undefined%
    \setlength{\unitlength}{69.35bp}%
    \ifx\svgscale\undefined%
      \relax%
    \else%
      \setlength{\unitlength}{\unitlength * \real{\svgscale}}%
    \fi%
  \else%
    \setlength{\unitlength}{\svgwidth}%
  \fi%
  \global\let\svgwidth\undefined%
  \global\let\svgscale\undefined%
  \makeatother%
  \begin{picture}(1,1.38428262)%
    \put(0,0){\includegraphics[width=\unitlength]{jwproj.pdf}}%
  \end{picture}%
\endgroup%
}}_n$ as shown in Figure \ref{fig:tcompose} below. 

\begin{figure}[H]
\centering
    \def\svgwidth{.7\columnwidth}
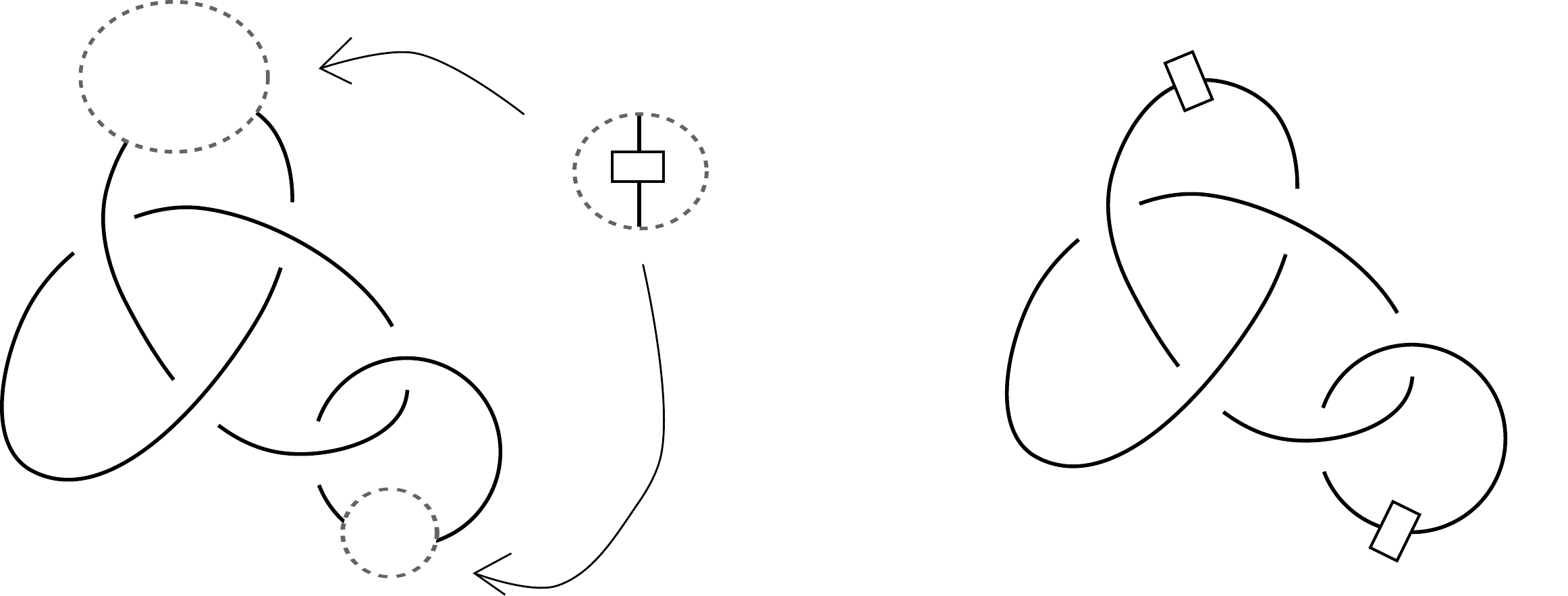
\caption{\label{fig:tcompose}}
\end{figure}  

The homology thus defined is invariant under Reidemeister moves of type II and III shown below but has a degree-shift under the Reidemeister move of type I. 

\begin{figure}[ht]
\def\svgwidth{\columnwidth}
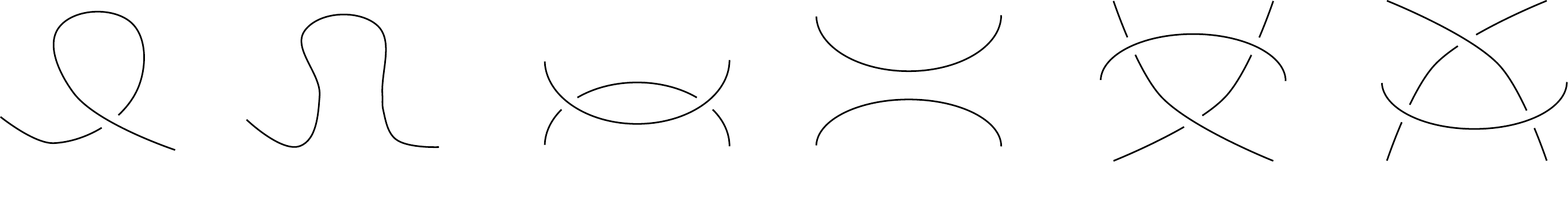
\end{figure}

\begin{equation} 
 \label{eq:twist} 
 \def\svgwidth{.5\columnwidth}
\begingroup%
  \makeatletter%
  \providecommand\color[2][]{%
    \errmessage{(Inkscape) Color is used for the text in Inkscape, but the package 'color.sty' is not loaded}%
    \renewcommand\color[2][]{}%
  }%
  \providecommand\transparent[1]{%
    \errmessage{(Inkscape) Transparency is used (non-zero) for the text in Inkscape, but the package 'transparent.sty' is not loaded}%
    \renewcommand\transparent[1]{}%
  }%
  \providecommand\rotatebox[2]{#2}%
  \ifx\svgwidth\undefined%
    \setlength{\unitlength}{419.94379883bp}%
    \ifx\svgscale\undefined%
      \relax%
    \else%
      \setlength{\unitlength}{\unitlength * \real{\svgscale}}%
    \fi%
  \else%
    \setlength{\unitlength}{\svgwidth}%
  \fi%
  \global\let\svgwidth\undefined%
  \global\let\svgscale\undefined%
  \makeatother%
  \begin{picture}(1,0.19540721)%
    \put(0,0){\includegraphics[width=\unitlength]{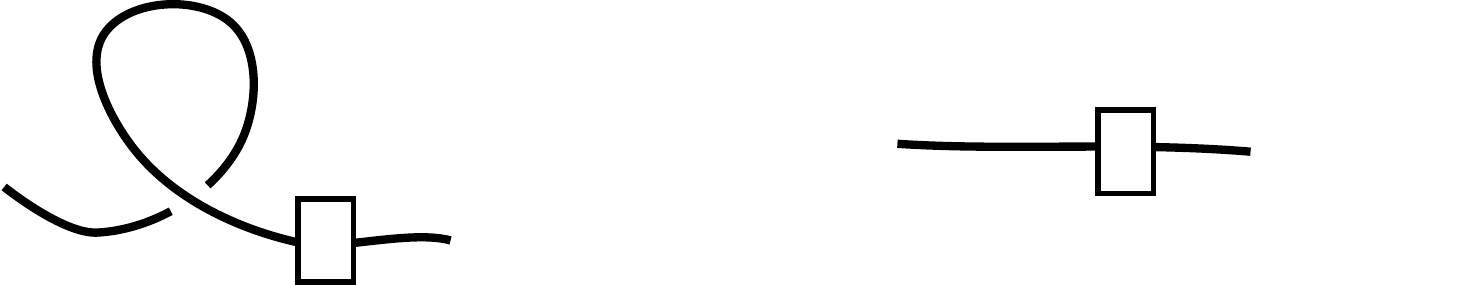}}%
    \put(0.26667477,0.05626132){\color[rgb]{0,0,0}\makebox(0,0)[lb]{\smash{$a$}}}%
    \put(0.33325854,0.10552562){\color[rgb]{0,0,0}\makebox(0,0)[lb]{\smash{$=$}}}%
    \put(0.38519187,0.10622326){\color[rgb]{0,0,0}\makebox(0,0)[lb]{\smash{$\mathbf{h}^{-\frac{1}{2}a^2}\mathbf{q}^{-a}$}}}%
    \put(0.81531965,0.11715007){\color[rgb]{0,0,0}\makebox(0,0)[lb]{\smash{$a$}}}%
  \end{picture}%
\endgroup%

 \end{equation}

We will not use the exact definition of a TQFT that will give the desired homology groups. The reader may consult \cite[Definition 7.1]{Bar05} for the precise definitions of such a TQFT. We denote the resulting complex by $\db{D^n_{\vcenter{\hbox{\def\svgwidth{.020\columnwidth} 
\begingroup%
  \makeatletter%
  \providecommand\color[2][]{%
    \errmessage{(Inkscape) Color is used for the text in Inkscape, but the package 'color.sty' is not loaded}%
    \renewcommand\color[2][]{}%
  }%
  \providecommand\transparent[1]{%
    \errmessage{(Inkscape) Transparency is used (non-zero) for the text in Inkscape, but the package 'transparent.sty' is not loaded}%
    \renewcommand\transparent[1]{}%
  }%
  \providecommand\rotatebox[2]{#2}%
  \ifx\svgwidth\undefined%
    \setlength{\unitlength}{69.35bp}%
    \ifx\svgscale\undefined%
      \relax%
    \else%
      \setlength{\unitlength}{\unitlength * \real{\svgscale}}%
    \fi%
  \else%
    \setlength{\unitlength}{\svgwidth}%
  \fi%
  \global\let\svgwidth\undefined%
  \global\let\svgscale\undefined%
  \makeatother%
  \begin{picture}(1,1.38428262)%
    \put(0,0){\includegraphics[width=\unitlength]{jwproj.pdf}}%
  \end{picture}%
\endgroup%
}}}}$, the underlying bi-graded chain groups by $\{C_{i, j}(D, n)\}$, and the resulting homology by $\{\Kh_{i, j}(D, n)\}$. From Definition \ref{defn:cp} of the colored Jones polynomial it is not hard to see that 
\begin{equation} \label{eqn:cjp} J_K(q; n)= ((-1)^nq^{\frac{n^2+2n}{2}})^{\omega(D)}\sum_{i, j}(-1)^j q^{i+j}\text{ dim } \Kh_{i, j}(D, n).  \end{equation}
Note that the presence of the projector means that we are allowing infinite complexes bounded from below. 

In general, for any skein $\Sk$ in $\Sk(\mathbb{R}^2)$ which may be decorated by the Jones-Wenzl projector, we denote the chain complex under the TQFT by $\db{\Sk}$ with underlying chain groups  $\{\Ckh_{i, j}(\Sk)\}$ and the resulting homology groups $\{\Kh_{i, j}(\Sk)\}$.  In a figure we will often depict the complex $\db{\Sk}$ by just the skein $\Sk$ without the double brackets.

\section{The tail homology of the colored Jones polynomial} \label{sec:tailhomology}

This section contains an exposition of Rozansky's work where the results are taken directly, or summarized, from \cite{Roz12}, with obvious adjustments for the change of variables. The goal is to give a precise definition of the tail homology which he has constructed. 

\subsection{Shifted homology} \label{subsec:shifted}
Let $D$ be a skein in $\Sk(\mathbb{R}^2)$ possibly with crossings and maybe decorated by Jones-Wenzl projectors, Rozansky defines a \emph{shifted homology} in \cite{Roz12} which we write here with a different notation for the degree shifts. 
\begin{defn}
\[\widetilde{\Kh}(D, n) := \textbf{h}^{\frac{1}{2}n^2c(D)} \textbf{q}^{n|s_A(D)|}\Kh(D, n).\]
\end{defn}
This is identical to his definition in \cite{Roz12} given below 
\[\widetilde{\Kh}(D_n) := \textbf{h}^{\frac{1}{2}n_{\times}(D_n)} \textbf{q}^{n_{\circ}(D_n)}\Kh(D_n), \]
where for him $D_n$ is the $n$-blackboard cable of $D$ decorated by the Jones-Wenzl projector. For a skein $D$,  the number $n_{\times}(D)$ which is the total number of single line crossings in $D$ (as he defines in \cite{Roz12}) is the same as the number of crossings in $D$. The quantity $n_{\circ}(D)$, which he defines as the total number of circles resulting from replacing the Jones-Wenzl projectors with identity braids and performing the $B$-splicings on all crossings, is also  the same as the number of state circles in the all-$B$ state of $D$. Adjusting for our convention for the variable changes of $\mathbf{q} \mapsto \mathbf{q}^{-1}$ and $\mathbf{h} \mapsto \mathbf{h}^{-1}$ gives the corresponding quantities concerning the all-$A$ Kauffman state.

\subsection{Tail homology} \label{subsec:tail}

Rozansky's main result in \cite{Roz12} is the following. 
\begin{thm}[{\cite[Theorem 2.13]{Roz12}}] \label{thm:tail-map} For any link diagram $D$ there is a sequence of degree-preserving maps 
\begin{equation} \widetilde{\Kh}(D, n) \stackrel{f_n}{\longrightarrow} \widetilde{\Kh}(D, n+1),  \label{eq:dfuncts} \end{equation}
which are isomorphisms on $\widetilde{\Kh_{i, *}}$ for $i\leq n-1$.  
\end{thm} 

This implies that the directed system formed by the maps $f_n$ has a limit. 

\begin{defn}
The \emph{tail homology} $H^{\infty}(D)$ of a link diagram $D$ is the limit of the direct system $\{\widetilde{H^{Kh}}(D, n), f_n\}$ formed by the maps $f_n$ of Theorem \ref{thm:tail-map}. 

\[H^{\infty}(D) = \lim_{\rightarrow} \widetilde{\Kh}(D, n). \] 
\end{defn} 

\begin{cor} \label{cor:directlimit}
The $i$th homology group of $H^{\infty}(D)$ is isomorphic to $\widetilde{\Kh_{i,*}}(D; i+1)$.
\end{cor} 

In addition, Rozansky proves the following bounds on shifted homology. Note again that what he means by a 	``uni-colored diagram $D_n$" is the same as $D^n_{\vcenter{\hbox{\def\svgwidth{.020\columnwidth} 
\begingroup%
  \makeatletter%
  \providecommand\color[2][]{%
    \errmessage{(Inkscape) Color is used for the text in Inkscape, but the package 'color.sty' is not loaded}%
    \renewcommand\color[2][]{}%
  }%
  \providecommand\transparent[1]{%
    \errmessage{(Inkscape) Transparency is used (non-zero) for the text in Inkscape, but the package 'transparent.sty' is not loaded}%
    \renewcommand\transparent[1]{}%
  }%
  \providecommand\rotatebox[2]{#2}%
  \ifx\svgwidth\undefined%
    \setlength{\unitlength}{69.35bp}%
    \ifx\svgscale\undefined%
      \relax%
    \else%
      \setlength{\unitlength}{\unitlength * \real{\svgscale}}%
    \fi%
  \else%
    \setlength{\unitlength}{\svgwidth}%
  \fi%
  \global\let\svgwidth\undefined%
  \global\let\svgscale\undefined%
  \makeatother%
  \begin{picture}(1,1.38428262)%
    \put(0,0){\includegraphics[width=\unitlength]{jwproj.pdf}}%
  \end{picture}%
\endgroup%
}}}$ considered in this paper. 
\begin{thm}{\cite[Theorem 2.12]{Roz12}} \label{thm:bounds} The shifted homology has a bound: $\widetilde{\Kh_{i, j}}(D, n) = 0$ if one of the following conditions is satisfied. 
\begin{align*}
&j< -\frac{1}{2}i-\frac{1}{2}c(D)-\frac{3}{2}c(D)^{\ell} \\
&j< -i - c(D)^{\ell}, 
\end{align*}
where $c(D)^{\ell}$ is the number of crossings in $D$ whose corresponding segments in the all-$A$ state are one-edged loops. Moreover, if $D$ is $A$-adequate, then 
\[\text{dim} \ \widetilde{\Kh_{i, -i}} = \begin{cases} &0, \text{ if } i > 0 \\ &1, \text{ if } i = 0\end{cases}. \]
\end{thm} 

The graded Euler characteristic of the tail homology is 
\[J_{D, \infty}(q) = \sum_{i, j} (-1)^j q^{i+j} \ \text{dim} \ H^{\infty}_{i, j} (D).  \] 

Because of these homological bounds, $J_{D, \infty}(q)$ is well-defined.

Suppose $K$ is $A$-adequate. Then $H^{\infty}(D)$ is independent of the $A$-adequate diagram chosen, and so is $J_{D, \infty}(q)$. Let $c(K)$ be the minimal crossing number among $A$-adequate diagrams representing $K$, then the inequalities of Theorem \ref{thm:bounds} translate to 
\begin{align*}
&j< -\frac{1}{2}i-\frac{1}{2}c(K) \\
&j< -i.
\end{align*}
In particular, let $D$ be an $A$-adequate diagram of $K$. The homology group $\widetilde{\Kh_{i, j}}(D, n) = 0$ if 
\[i+j< \frac{1}{2}i-\frac{1}{2}c(K). \] 

In view of \eqref{eqn:cjp}, let $n >> c(K)$. If $i \geq n$, then $\widetilde{\Kh_{i, j}}(D, n) = 0$ for $D$ if $i+j < \frac{1}{2}n-\frac{1}{2}c(K)$. Therefore, the terms of the graded Euler characteristic: 
\[ \sum_{i, j}(-1)^j q^{i+j}\text{ dim } \widetilde{\Kh_{i, j}}(D, n), \] where $i+j < \frac{1}{2}n-\frac{1}{2}c(K)$ are determined by $J_{D, \infty}(q)$ up to multiplication by a power of $q$. This is the following corollary from \cite{Roz12}, which we adjust for unframed links by multiplying by a power of $q$. 

\begin{thm}{\cite[Theorem 2.6]{Roz12}} \label{thm:adequatelower}
The graded Euler characteristic of the tail homology determines the lower powers of $q$ in the colored Jones polynomial of a $A$-adequate link: 
\[J_K(q; n)=\left((-1)^nq^{\frac{n^2+2n}{2}}\right)^{\omega(D)} q^{-\frac{1}{2}n^2c(D)-n|s_A(D)|}\left(J_{D, \infty}(q) + O(q^{\frac{1}{2}n-\frac{1}{2}c(K)}) \right).  \] 
\end{thm}

It is in this sense that the tail homology categorifies the tail for the colored Jones polynomial when the knot is $A$-adequate.
\section{Proof of Conjecture \ref{thm:main}} \label{sec:conj}

A substantial portion of our proof relies on interpreting and adapting the content of Rozansky's arguments \cite{Roz12}, which we will indicate and give original references whenever appropriate. 

We consider the following combinatorial data of a skein $\Sk\in TL$ with crossings which may or may not be decorated by Jones-Wenzl projectors. We also consider a more general Kauffman state $\sigma$, which is a choice of $A$- or $B$-resolution at a crossing as before, but restricted to a specified subset of the crossings of $\Sk$. It should be assumed that a Kauffman state is applied to all the crossings of $\Sk$ unless otherwise indicated. 

\begin{itemize}
\item $C(\Sk):=$ The set of crossings in $\Sk$, and let $c(\Sk) = |C(\Sk)|$ be the number of crossings in $\Sk$.
\item $\sgn(\sigma) := \sgn_B(\sigma) - \sgn_A(\sigma),$ where
\begin{align*} 
\sgn_A(\sigma) &:= \frac{1}{2} (\# \text{ of crossings where the $A$-resolution is chosen in $\sigma$)}\text{, and}  \\ 
\sgn_B(\sigma) &:= \frac{1}{2} (\# \text{ of crossings where the $B$-resolution is chosen in $\sigma$}).\end{align*}
\end{itemize} 
\end{defn} 

\subsection{Tools for working with the chain complex}
We will use the following lemmas. Recall from Definition \ref{defn:mhd} that $|\mathbf{T}|$ is the minimum homological degree of the complex $(\mathbf{T}, d_{\mathbf{T}})$. 

\begin{lem}{\cite[Theorem 2.10]{Roz12}} \label{lem:hoestimate}
Let $\Sk$ be a skein with crossings which may include disjoint circles and may be decorated by Jones-Wenzl projectors. Let $c(\Sk)$ be the total number of crossings of $\Sk$. The minimum homological degree of the complex $\db{\Sk}$ is bounded below by $-\frac{1}{2}c(\Sk)$. i.e., 
\[ |\db{\Sk}| \geq -\frac{1}{2}c(\Sk).\] 
\end{lem} 

\begin{thm}{\cite[Theorem 3.3]{Roz12}} \label{lem:psimplify} Let $\beta$ be a braid shown below in Figure \ref{fig:straightbraid}, with the braid strands going from left to right. Let a crossing of $\beta$ be positive or negative depending on whether it is $\vcenter{\hbox{\includegraphics[scale=.2]{crossing1.png}}}$ or $\vcenter{\hbox{\includegraphics[scale=.2]{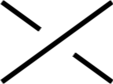}}}$, respectively. We have the following homotopy equivalence. 
\begin{figure}[ht]
\centering
    \def\svgwidth{.7\columnwidth}
    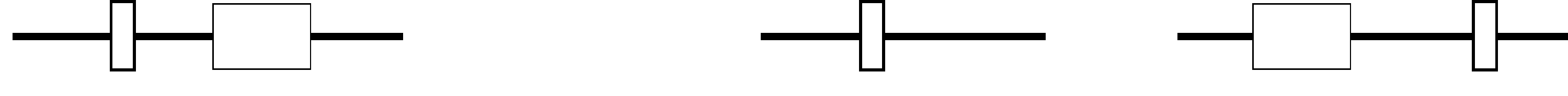,
    \caption{\label{fig:straightbraid}}
    \end{figure}
    
where $n_+$(resp. $n_-$) is the number of positive (resp.) negative crossings of the braid.
\end{thm}

\begin{lem} \label{lem:coneswitch} If a chain complex $\mathbf{T}$ has the following mapping cone presentation
\[\mathbf{T} \sim Cone(\mathbf{A} \stackrel{f}{\rightarrow} Cone(\mathbf{B}\stackrel{g}{\rightarrow} \mathbf{C})). \] Then it also admits the mapping cone presentation 
\[\mathbf{T}\sim Cone(Cone(\mathbf{A}[-1]\stackrel{f_{\mathbf{B}}}{\rightarrow} \mathbf{B}) \stackrel{f_{\mathbf{C}}\oplus g}{\rightarrow} \mathbf{C}), \]
where $f_{\mathbf{B}}$ is the component of $f$ in $\mathbf{B}$ and $f_{\mathbf{C}}$ is the component in $\mathbf{C}$. 
\end{lem} 

\begin{proof}
This is immediate from writing out the chain complexes of the two presentations as direct sums and comparing the differentials. 
\end{proof}

\subsection{Decomposition over Kauffman states} 

Let $D$ be a non-$A$ adequate link diagram, which means that its all-$A$ state has a loop $\ell$. Let $\ell^{n+1}$ denote the $n+1$-cabled crossings corresponding to $\ell$ in $D^{n+1}$. Without loss of generality, we assume that $\ell$ lies in the bounded component of a state circle to which the corresponding one-edged loop is attached. Let $C(D)$ be the set of crossings of a diagram $D$, we consider Kauffman states on $D^{n+1}$ restricted to the set of crossings $C(D^{n+1}) \setminus \ell^{n+1}$. 

We have the skein exact sequence on homology groups from choosing the $A$- or $B$-resolution at each crossing via the map
\[ \db{\vcenter{\hbox{\includegraphics[scale=.2]{crossing1.png}}}} = \fbox{$\mathbf{h}^{\frac{1}{2}}\db{\vcenter{\hbox{\includegraphics[scale=.2]{crossing2.png}}}} \ \stackrel{s}{\rightarrow} \textbf{h}^{-\frac{1}{2}} \ \db{\vcenter{\hbox{\includegraphics[scale=.2]{crossing3.png}}}}$}.\] Pick an  order on the crossings in $C(D^{n+1}) \setminus \ell^{n+1}$, we paste together the exact sequence for every crossing to obtain the presentation of $\db{D^{n+1}_{\vcenter{\hbox{\def\svgwidth{.020\columnwidth} 
\begingroup%
  \makeatletter%
  \providecommand\color[2][]{%
    \errmessage{(Inkscape) Color is used for the text in Inkscape, but the package 'color.sty' is not loaded}%
    \renewcommand\color[2][]{}%
  }%
  \providecommand\transparent[1]{%
    \errmessage{(Inkscape) Transparency is used (non-zero) for the text in Inkscape, but the package 'transparent.sty' is not loaded}%
    \renewcommand\transparent[1]{}%
  }%
  \providecommand\rotatebox[2]{#2}%
  \ifx\svgwidth\undefined%
    \setlength{\unitlength}{69.35bp}%
    \ifx\svgscale\undefined%
      \relax%
    \else%
      \setlength{\unitlength}{\unitlength * \real{\svgscale}}%
    \fi%
  \else%
    \setlength{\unitlength}{\svgwidth}%
  \fi%
  \global\let\svgwidth\undefined%
  \global\let\svgscale\undefined%
  \makeatother%
  \begin{picture}(1,1.38428262)%
    \put(0,0){\includegraphics[width=\unitlength]{jwproj.pdf}}%
  \end{picture}%
\endgroup%
}}}}$ as an iterated mapping cone. Let $\Sk_{\sigma}$ be a skein with crossings and decorated by Jones-Wenzl projectors resulting from a state $\sigma$. Then the chain groups of $\db{D^{n+1}_{\vcenter{\hbox{\def\svgwidth{.020\columnwidth} 
\begingroup%
  \makeatletter%
  \providecommand\color[2][]{%
    \errmessage{(Inkscape) Color is used for the text in Inkscape, but the package 'color.sty' is not loaded}%
    \renewcommand\color[2][]{}%
  }%
  \providecommand\transparent[1]{%
    \errmessage{(Inkscape) Transparency is used (non-zero) for the text in Inkscape, but the package 'transparent.sty' is not loaded}%
    \renewcommand\transparent[1]{}%
  }%
  \providecommand\rotatebox[2]{#2}%
  \ifx\svgwidth\undefined%
    \setlength{\unitlength}{69.35bp}%
    \ifx\svgscale\undefined%
      \relax%
    \else%
      \setlength{\unitlength}{\unitlength * \real{\svgscale}}%
    \fi%
  \else%
    \setlength{\unitlength}{\svgwidth}%
  \fi%
  \global\let\svgwidth\undefined%
  \global\let\svgscale\undefined%
  \makeatother%
  \begin{picture}(1,1.38428262)%
    \put(0,0){\includegraphics[width=\unitlength]{jwproj.pdf}}%
  \end{picture}%
\endgroup%
}}}}$ is a direct sum of the chain groups of $\Sk$.
\[\db{D^{n+1}_{\vcenter{\hbox{\def\svgwidth{.020\columnwidth} 
\begingroup%
  \makeatletter%
  \providecommand\color[2][]{%
    \errmessage{(Inkscape) Color is used for the text in Inkscape, but the package 'color.sty' is not loaded}%
    \renewcommand\color[2][]{}%
  }%
  \providecommand\transparent[1]{%
    \errmessage{(Inkscape) Transparency is used (non-zero) for the text in Inkscape, but the package 'transparent.sty' is not loaded}%
    \renewcommand\transparent[1]{}%
  }%
  \providecommand\rotatebox[2]{#2}%
  \ifx\svgwidth\undefined%
    \setlength{\unitlength}{69.35bp}%
    \ifx\svgscale\undefined%
      \relax%
    \else%
      \setlength{\unitlength}{\unitlength * \real{\svgscale}}%
    \fi%
  \else%
    \setlength{\unitlength}{\svgwidth}%
  \fi%
  \global\let\svgwidth\undefined%
  \global\let\svgscale\undefined%
  \makeatother%
  \begin{picture}(1,1.38428262)%
    \put(0,0){\includegraphics[width=\unitlength]{jwproj.pdf}}%
  \end{picture}%
\endgroup%
}}}} = \bigoplus_{\sigma \text{ on } C(D^{n+1}) \setminus \ell^{n+1}} \db{\Sk_{\sigma}}. \] 

\begin{lem} \label{lem:decompstate} The unshifted homology group $\Kh_{i, *}(D, n+1) = 0$ if $\Kh_{i-sgn(\sigma), *} (\Sk_{\sigma}) = 0$ for all $\sigma$.    
\end{lem}
\begin{proof}
We have the long exact sequence 
\[\cdots \rightarrow \Kh_{i-\frac{1}{2}, j}(\vcenter{\hbox{\includegraphics[scale=.2]{crossing2.png}}}) \rightarrow  \Kh_{i, j}(\vcenter{\hbox{\includegraphics[scale=.2]{crossing1.png}}}) \rightarrow  \Kh_{i+\frac{1}{2}, j}(\vcenter{\hbox{\includegraphics[scale=.2]{crossing3.png}}}) \rightarrow \cdots. \]
We have another long exact sequence with the remaining crossings with each of  $ \Kh_{i-\frac{1}{2}, j}(\vcenter{\hbox{\includegraphics[scale=.2]{crossing2.png}}})$ and $\Kh_{i+\frac{1}{2}, j}(\vcenter{\hbox{\includegraphics[scale=.2]{crossing3.png}}})$. We repeat this process, obtaining an exact sequence for every choice of $A$- or $B$-resolution, until all the crossings in $C(D^n)\setminus \ell^n$ are resolved. 
\begin{center}
\begin{figure}[ht]

\includegraphics[scale=.6]{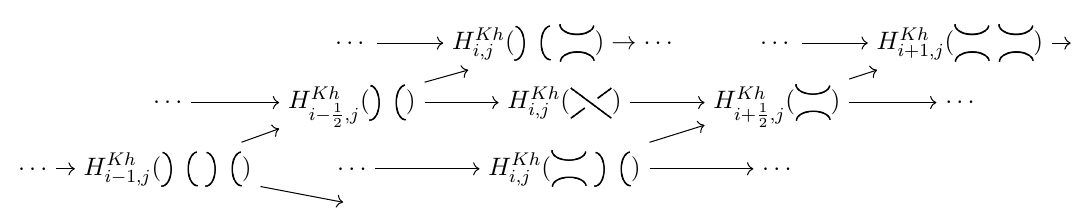}
 \caption{The pasted exact sequences from resolving two crossings is shown.}
 \end{figure}
 \end{center}
Each time a resolution is chosen, the homological degree of the constituent homology group in the exact sequence is shifted by opposite the sign of the choice of the resolution. If each of these homology groups is trivial, then the original homology group is trivial. 

\end{proof}

\begin{lem} \label{lem:disjointsk} Let $D^{n+1}_{\vcenter{\hbox{\def\svgwidth{.020\columnwidth} 
\begingroup%
  \makeatletter%
  \providecommand\color[2][]{%
    \errmessage{(Inkscape) Color is used for the text in Inkscape, but the package 'color.sty' is not loaded}%
    \renewcommand\color[2][]{}%
  }%
  \providecommand\transparent[1]{%
    \errmessage{(Inkscape) Transparency is used (non-zero) for the text in Inkscape, but the package 'transparent.sty' is not loaded}%
    \renewcommand\transparent[1]{}%
  }%
  \providecommand\rotatebox[2]{#2}%
  \ifx\svgwidth\undefined%
    \setlength{\unitlength}{69.35bp}%
    \ifx\svgscale\undefined%
      \relax%
    \else%
      \setlength{\unitlength}{\unitlength * \real{\svgscale}}%
    \fi%
  \else%
    \setlength{\unitlength}{\svgwidth}%
  \fi%
  \global\let\svgwidth\undefined%
  \global\let\svgscale\undefined%
  \makeatother%
  \begin{picture}(1,1.38428262)%
    \put(0,0){\includegraphics[width=\unitlength]{jwproj.pdf}}%
  \end{picture}%
\endgroup%
}}}$ be the $n$-cable of a non $A$-adequate link diagram $D$ decorated by a Jones-Wenzl projector and let $\sigma$ be a Kauffman state on $C(D^{n+1} \setminus \ell^{n+1})$ with $\ell^{n+1}$ the set of $n$-cabled crossings corresponding to the loop $\ell$ of $s_A(D)$. 
The skein $\Sk_{\sigma}$ obtained by applying $\sigma$ to $D^n_{\vcenter{\hbox{\def\svgwidth{.020\columnwidth} 
\begingroup%
  \makeatletter%
  \providecommand\color[2][]{%
    \errmessage{(Inkscape) Color is used for the text in Inkscape, but the package 'color.sty' is not loaded}%
    \renewcommand\color[2][]{}%
  }%
  \providecommand\transparent[1]{%
    \errmessage{(Inkscape) Transparency is used (non-zero) for the text in Inkscape, but the package 'transparent.sty' is not loaded}%
    \renewcommand\transparent[1]{}%
  }%
  \providecommand\rotatebox[2]{#2}%
  \ifx\svgwidth\undefined%
    \setlength{\unitlength}{69.35bp}%
    \ifx\svgscale\undefined%
      \relax%
    \else%
      \setlength{\unitlength}{\unitlength * \real{\svgscale}}%
    \fi%
  \else%
    \setlength{\unitlength}{\svgwidth}%
  \fi%
  \global\let\svgwidth\undefined%
  \global\let\svgscale\undefined%
  \makeatother%
  \begin{picture}(1,1.38428262)%
    \put(0,0){\includegraphics[width=\unitlength]{jwproj.pdf}}%
  \end{picture}%
\endgroup%
}}}$ contains a cap or a cup composed with a Jones-Wenzl projector unless it is isotopic to a disjoint union of a skein $\Sk_{\sigma}^k$ with circles, where $\Sk_{\sigma}^k$ has the form as shown in Figure \ref{fig:gform}. 
\begin{figure}[ht]
\centering
    \def\svgwidth{.3\columnwidth}
    \input{gform.pdf_tex}
    \caption{\label{fig:gform} $\Sk_{\sigma}^k$, with $0\leq k \leq n+1$.}
\end{figure}
\end{lem} 

\begin{proof}

The skein $\Sk_{\sigma}$ may be written as a composition of the skeins of two disks $\Sk(D_1, 2(n+1))$, $\Sk(D_2, 2(n+1))$, one containing the $n$-cabled crossing, and the other containing non-intersecting strands between the projectors, see Figure \ref{fig:gformd}.  

\begin{figure}[H]
\begin{center}
\def\svgwidth{.5\columnwidth}
    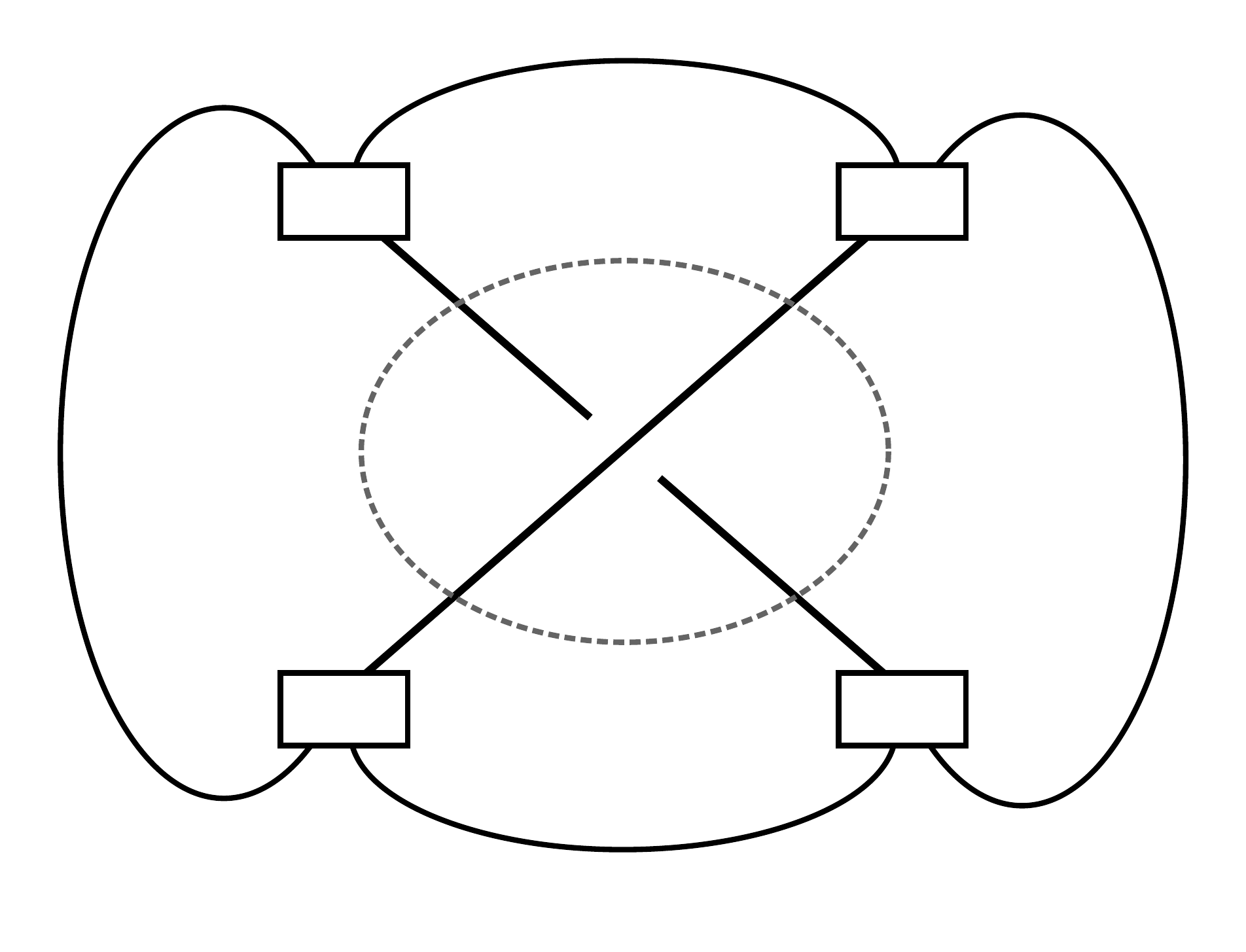
    \caption{\label{fig:gformd}}
    \end{center}
\end{figure}

Configurations of these non-intersecting strands in $D_2$ are determined by the presence of a cap/cup composed with a projector. A configuration containing a cap/cup composed with a projector will result in a homotopically trivial complex. It is straightforward to see that a non-trivial configuration has to have the form indicated in Figure \ref{fig:gform} as follows. Here we flatten out the disk $D_2$ and denote the four projectors by tl (top-left), tr (top-right), ll (lower-left), and  lr (lower-right) as shown in Figure \ref{fig:diskarg}. 

\begin{figure}[H]
\begin{center}
\def\svgwidth{.7\columnwidth}
    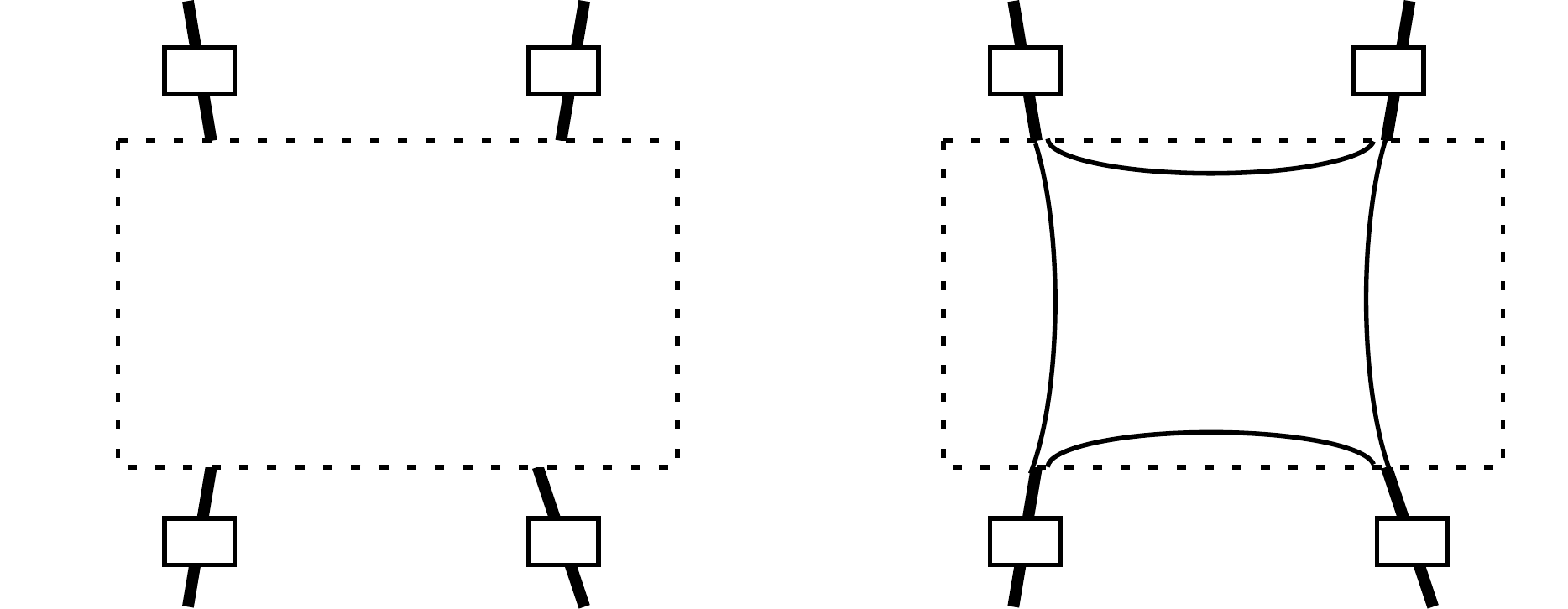
    \caption{\label{fig:diskarg}}
    \end{center}
    \end{figure}

We first argue that, from the perspective shown for $D_2$ in Figure \ref{fig:diskarg}, we cannot have arcs connecting the tl projector to the lr projector, since that will necessarily result in a cap/cup composed with the tr or ll projector by comparing the number of strands to the available endpoints. This also uses the fact that the strands cannot intersect. There cannot be arcs connecting the tr projector to the ll projector for the same reason. Thus we can only have arcs between the tr and lr projectors, and arcs between the tl and ll projectors. The number of strands between each such pair is denoted by $k, k'$, respectively. If $k \not= k'$ then again there is a cap/cup composed with a projector. Therefore $k = k'$ and we come to the form in the lemma. 
\end{proof}

\subsection{Long exact sequences of homology from local transformations}
A \emph{local replacement}, as it is called in \cite{Roz12}, is a triple $M=(\tau_i, \tau_c, \tau_f)$ of tangle diagrams decorated by the Jones-Wenzl projector. Let a triple of diagrams $(D_i, D_c, D_f)$ differ locally in the skeins $(\tau_i, \tau_c, \tau_f)$. 

Let $n_f = \frac{1}{2}c(D_f)-\frac{1}{2}c(D_i)$, $m_f = |s_A(D_f)|-|s_A(D_i)|$ (here we take the $A$-resolution on \emph{all} of the crossings of $D_f$, $D_i$), and $\tau_f' = \mathbf{h}^{n_f}\mathbf{q}^{m_f} \tau_f$. 
If there is a mapping cone presentation of $\tau_i$: 
\begin{equation} \db{\tau_i} \sim \fbox{$\db{\tau_c} \stackrel{f}{\rightarrow} \db{\tau_f'}$}.  \label{eqn:lcmcone}\end{equation} 
Then we also have the mapping cone presentation 
\begin{equation} \db{\tau_c} = \fbox{$\mathbf{h}\db{\tau'_f} \stackrel{g}{\rightarrow} \db{\tau_i}$}  \label{eqn:lcmcone}\end{equation} from the natural map $\tau_f' \stackrel{g}{\rightarrow} \tau_i$. 
Here we indicate explicitly the map $f$ in the mapping cone presentation for $\db{\tau_i}$. Otherwise, the notation agrees with that of \cite[Pg. 554, (3.1)]{Roz12}. A \emph{local transformation} is a local replacement $(\tau_i, \tau_c, \tau_f)$ along with a map $g:\tau_f' \rightarrow \tau_i$ specifying the mapping cone presentation \eqref{eqn:lcmcone}.

With the natural map $\tau_i \rightarrow \tau_c$, this induces a long exact sequence of homology groups.  
\begin{equation} \mathbf{h}^{-1}\Kh(D_c') \stackrel{f}{\rightarrow} \widetilde{\Kh}(D_f) \stackrel{g}{\rightarrow} \widetilde{\Kh}(D_i) \rightarrow \Kh(D_c'), \label{eqn:LES} \end{equation} 
where $H^{Kh}(D_c') = \mathbf{h}^{\frac{1}{2}c(D_i)} \mathbf{q}^{|s_A(D_i)|} H^{Kh}(D_c).$ 
The long exact sequence implies the following.　
\begin{thm}{\cite[Proposition 3.1]{Roz12}} \label{thm:hles}
 If $\Kh_{i, *}(D_c)=0$ for $i \leq m-1$, then the degree-preserving map 
\[ \widetilde{\Kh}(D_f) \stackrel{g}{\rightarrow} \widetilde{\Kh}(D_i),  \] is an isomorphism on  $\widetilde{\Kh_{i, *}}(D_f)$ for 
\begin{equation}
i \leq m + \frac{1}{2}c(D_i)-2. 
\end{equation} 
\end{thm}

\subsection{Local transformations} \label{subsec:lt}

We exhibit three local transformations $M_I$, $M_{II}$, and $M_{III}$ shown below in Figure \ref{fig:localskein1}, \ref{fig:localskein2}, and  \ref{fig:localskein3}, and we apply Theorem \ref{thm:hles} to triples of diagrams corresponding to transformations. The local transformations are the same as those from \cite{Roz12} with obvious adjustments for the variable change in our conventions. However, we prove the statements for a different set of triples of diagrams. In particular,  we consider triples of diagrams from the skein $\Sk_{\sigma}^k$ differing locally in transformations $M_I, M_{II}$ and  $M_{III}$. Since the state $\sigma$ is not relevant for this part of the argument, we will suppress the notation indicating $\sigma$ and just refer to $\Sk^k$ for Section \ref{subsec:lt}-\ref{subsec:simpliflysk}.

To specify the location where we apply the local transformation, we will label the projectors of $\Sk^k$ as shown in the following figure. 

\begin{figure}[H]
\centering
\def\svgwidth{.4\columnwidth}
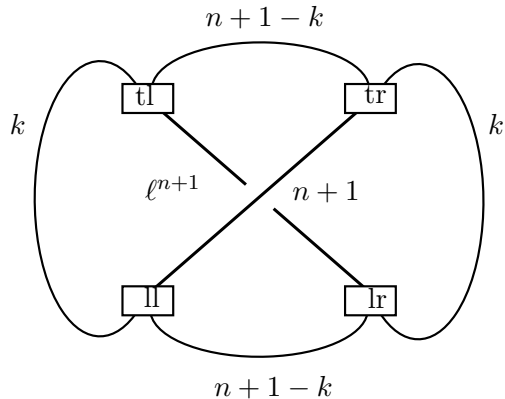
\caption{\label{fig:gforml} The top-left, top-right, lower-left, and lower-right projectors.}
\end{figure}

\subsection{Local transformation I} \ \\

\begin{prop}{\cite[Proposition 4.1]{Roz12}} \label{p:lcm1}
Let $\Sk^k$ be a skein of the form in Figure \ref{fig:gform}. Consider the local transformation $M_I = (\tau_i, \tau_c, \tau_f)$ on the $n$-cabled crossing in $\Sk^k$, with $(D_i, D_c, D_f)$ as shown in Figure \ref{fig:gformlcm1}.

\begin{figure}[H]
\centering
\def\svgwidth{.9\columnwidth}
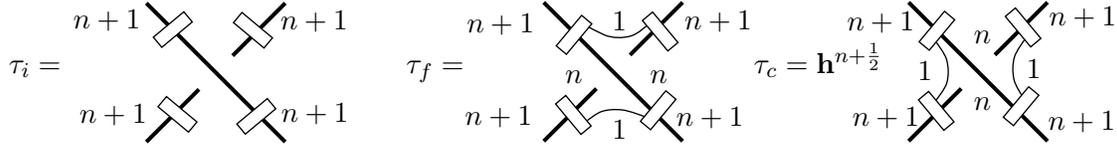
\caption{\label{fig:localskein1} Local transformation $M_{I}$. This is a version of the skein relation in Definition \ref{defn:skein}, see also \cite{Yam89}. \label{fig:ls1}}
\end{figure}
Then
\begin{enumerate}[(a)]
\item there is a map $f_{I}$ such that we have a mapping cone presentation 
\[\tau_i = \fbox{$\tau_c \stackrel{f_I}{\rightarrow} \tau_f'$},\]
where $\tau_f' = \mathbf{h}^{-n-\frac{1}{2}}\tau_f$. 

\item The induced degree-preserving map
\[ \widetilde{\Kh}(D_f) \stackrel{g_I}{\rightarrow} \widetilde{\Kh}(D_i) \] is an isomorphism on $\widetilde{\Kh_{i, *}}$ for $i\leq 2n-1$.

\begin{figure}[H]
\centering
\def\svgwidth{1\columnwidth}
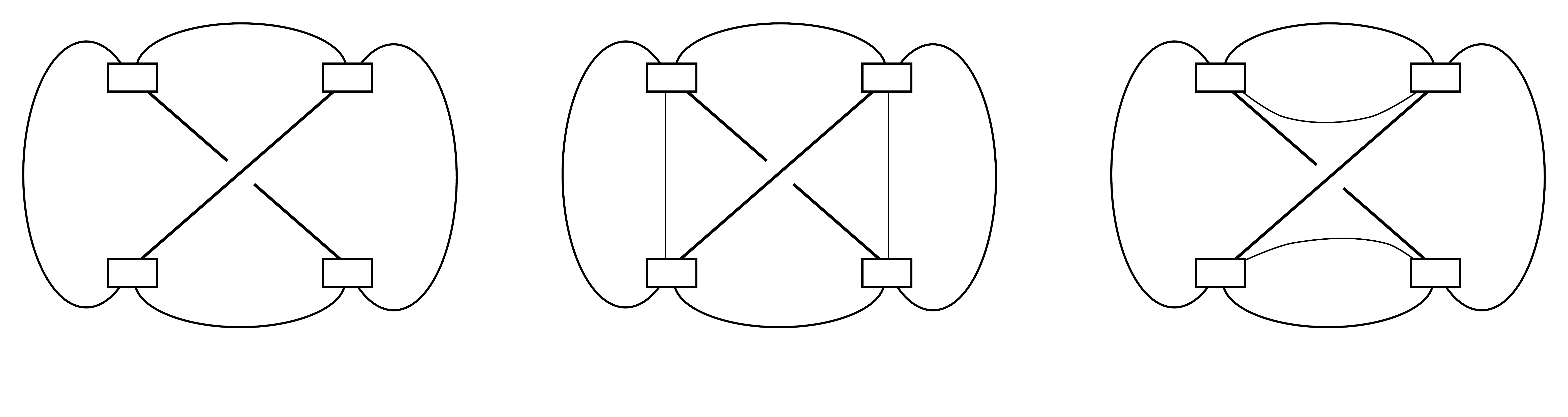
\caption{\label{fig:gformlcm1}}
\end{figure}

\end{enumerate}
\end{prop}
\begin{proof} \ \\
\begin{enumerate}[(a)]
\item The map inducing the mapping cone presentation is the saddle cobordism 
\[ \vcenter{\hbox{\includegraphics[scale=.2]{crossing1.png}}} = \text{Cone}\left(\textbf{h}^{-\frac{1}{2}} \ \vcenter{\hbox{\includegraphics[scale=.2]{crossing2.png}}} \ \stackrel{s}{\rightarrow} \textbf{h}^{-\frac{1}{2}} \ \vcenter{\hbox{\includegraphics[scale=.2]{crossing3.png}}}\right). \]

The argument follows from that of \cite[Proposition 4.1]{Roz12} without further adjustment except for the sign of the crossing since his arguments are completely local.

\item 
By Theorem \ref{thm:hles}, we need to show that $\Kh_{i, *}(D_c)=0$ for 
$i\leq 2n-\frac{1}{2}c(D_i)$. Note that $c(D_{c})=c(D_i)-2n-1$. By Lemma \ref{lem:hoestimate}, $H_{i, *}^{Kh}(D_c) = 0$ for 
\[ i <-\frac{1}{2}c(D_c) + \underbrace{n+\frac{1}{2}}_{\text{the 
$\mathbf{h}^{n+\frac{1}{2}}$ shifting on $D_c$}} = -\frac{1}{2}c(D_i) + 2n+1.\] 

\end{enumerate}
\end{proof} 

\subsection{Local transformation II}

Before showing the second local transformation we will use yet another mapping cone presentation for the Jones-Wenzl projector. The arguments in \cite{Roz12} are again completely local so we will refer the reader to \cite{Roz12} for the proofs.

\begin{thm}{\cite[Theorem 3.8]{Roz12}} \label{lem:jwidfinal}
The Jones-Wenzl projector has the following mapping cone presentation. 
\begin{figure}[H]
    \centering
    \def\svgwidth{.9\columnwidth}
    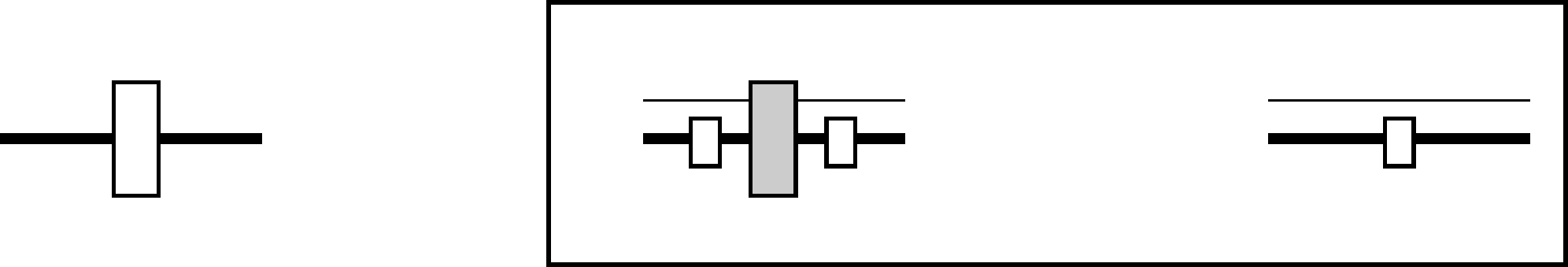, 
\end{figure}
where the first complex of the mapping cone is homotopic to a complex in which each individual term has the form indicated in the right-hand side of \eqref{eqn:tc1}.
\begin{figure}[H]
\begin{equation} \label{eqn:tc1}
\centering
\def\svgwidth{1.5\columnwidth}
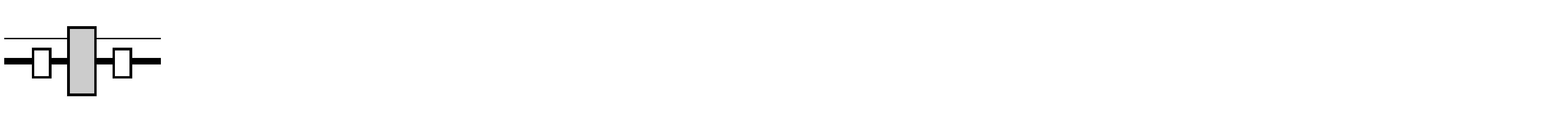
\end{equation}
\end{figure}
\end{thm} 

We would just like to remark that in Rozansky's proof, the map giving the mapping cone presentation comes from Lemma \ref{lem:jwidc12} and Lemma \ref{lem:jwmcone}.

Let $\tau_c = \def\svgwidth{.15\columnwidth} 
\begingroup%
  \makeatletter%
  \providecommand\color[2][]{%
    \errmessage{(Inkscape) Color is used for the text in Inkscape, but the package 'color.sty' is not loaded}%
    \renewcommand\color[2][]{}%
  }%
  \providecommand\transparent[1]{%
    \errmessage{(Inkscape) Transparency is used (non-zero) for the text in Inkscape, but the package 'transparent.sty' is not loaded}%
    \renewcommand\transparent[1]{}%
  }%
  \providecommand\rotatebox[2]{#2}%
  \ifx\svgwidth\undefined%
    \setlength{\unitlength}{100bp}%
    \ifx\svgscale\undefined%
      \relax%
    \else%
      \setlength{\unitlength}{\unitlength * \real{\svgscale}}%
    \fi%
  \else%
    \setlength{\unitlength}{\svgwidth}%
  \fi%
  \global\let\svgwidth\undefined%
  \global\let\svgscale\undefined%
  \makeatother%
  \begin{picture}(1,0.25673703)%
    \put(0,0){\includegraphics[width=\unitlength]{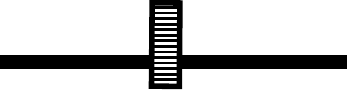}}%
  \end{picture}%
\endgroup%
_{n+1}$ be a shorthand notation for the complex of the right-hand side of the homotopy equivalence \eqref{eqn:tc1}, we can now present local transformation II.

\begin{prop}{\cite[Proposition 4.2]{Roz12}} \label{p:lcm2}
Let $\Sk^k_I$ be the skein obtained from $\Sk^k$ by an application of the local transformation I to the only $n$-cabled crossing as in Proposition \ref{p:lcm1}. Pick the top-right projector (recall the labelling in the caption of Figure \ref{fig:gforml}) and consider the triple of local transformation $M_{II} = (\tau_i, \tau_c, \tau_f)$ applied to the chosen projector of $\Sk^k_I$, with $(D_i, D_c, D_f)$ as shown in Figure \ref{fig:gformlcm2}.
\begin{figure}[H]
\centering
\def\svgwidth{\columnwidth}
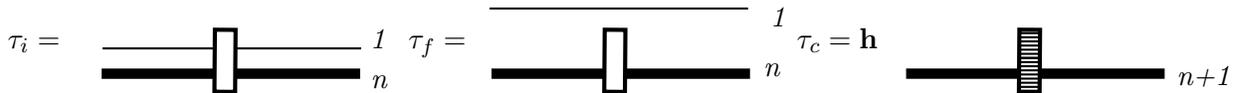
\caption{\label{fig:localskein2} Local transformation $M_{II}$. The complex $\db{\tau_c}$ comes from the first complex of the mapping cone presentation in Theorem \ref{lem:jwidfinal}.}
\end{figure}

Then 
\begin{enumerate}[(a)]
\item there is a map $f_{II}$ such that we have a mapping cone presentation 
\[\tau_i = \fbox{$\tau_c \stackrel{f_{II}}{\rightarrow} \tau_f'$},\]
where $\tau_f' = \tau_f $, and
\item \label{p:lcm2b} the induced degree-preserving map
\[ \widetilde{\Kh}(D_f) \stackrel{g_{II}}{\rightarrow} \widetilde{\Kh}(D_i) \] is an isomorphism on $\widetilde{\Kh_{i, *}}$ for $i\leq n-1$.
\begin{figure}[H]
\centering
\def\svgwidth{1\columnwidth}
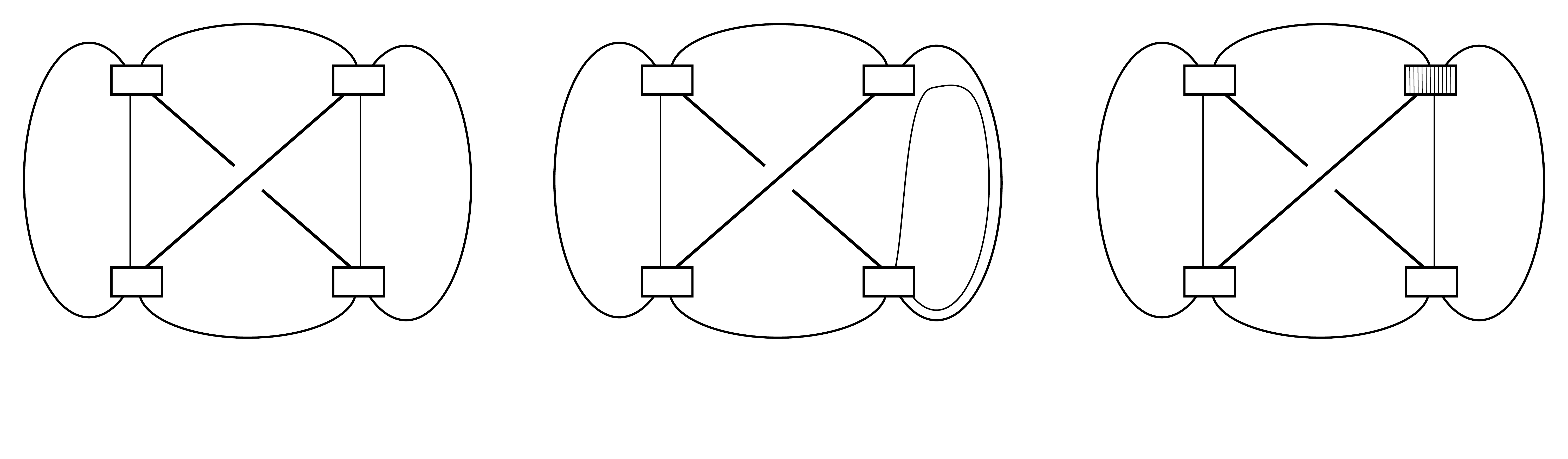
\caption{\label{fig:gformlcm2}}
\end{figure}
\end{enumerate}
Let $\tilde{S}^k_{I}$ be the skein $D_f$, and now let $D_i =\tilde{S}^k_{I}$,  we have the same statement for the triple $M_{II} = (\tau_i, \tau_c, \tau_f)$ applied to the lower-left  projector of $\tilde{S}^k_{I}$. 
\begin{figure}[H]
\centering
\def\svgwidth{1\columnwidth}
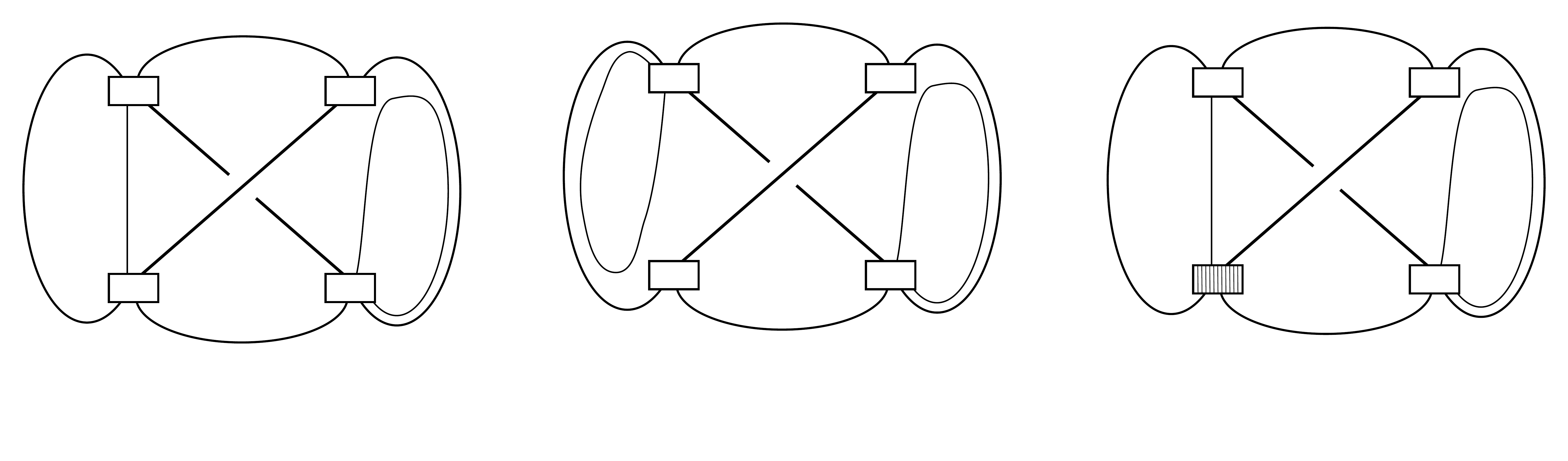
\end{figure}
\end{prop}

\begin{proof} \ \\
\begin{enumerate}[(a)] 
\item 
We have the desired mapping cone presentation by Theorem \ref{lem:jwidfinal}. Note that $c(D_f) = c(D_i)$ and $|s_A(D_f)| = |s_A(D_i)|$.
   
\item    
As in Rozansky's proof, there is the following equivalence by sliding a single strand along the projector for each non-trivial term in the complex $\def\svgwidth{.15\columnwidth} 
\begingroup%
  \makeatletter%
  \providecommand\color[2][]{%
    \errmessage{(Inkscape) Color is used for the text in Inkscape, but the package 'color.sty' is not loaded}%
    \renewcommand\color[2][]{}%
  }%
  \providecommand\transparent[1]{%
    \errmessage{(Inkscape) Transparency is used (non-zero) for the text in Inkscape, but the package 'transparent.sty' is not loaded}%
    \renewcommand\transparent[1]{}%
  }%
  \providecommand\rotatebox[2]{#2}%
  \ifx\svgwidth\undefined%
    \setlength{\unitlength}{100bp}%
    \ifx\svgscale\undefined%
      \relax%
    \else%
      \setlength{\unitlength}{\unitlength * \real{\svgscale}}%
    \fi%
  \else%
    \setlength{\unitlength}{\svgwidth}%
  \fi%
  \global\let\svgwidth\undefined%
  \global\let\svgscale\undefined%
  \makeatother%
  \begin{picture}(1,0.25673703)%
    \put(0,0){\includegraphics[width=\unitlength]{lcm3c.pdf}}%
  \end{picture}%
\endgroup%
$. The last step comes from Lemma \ref{lem:psimplify}.

\begin{figure}[H]
\centering
\def\svgwidth{.8\columnwidth}
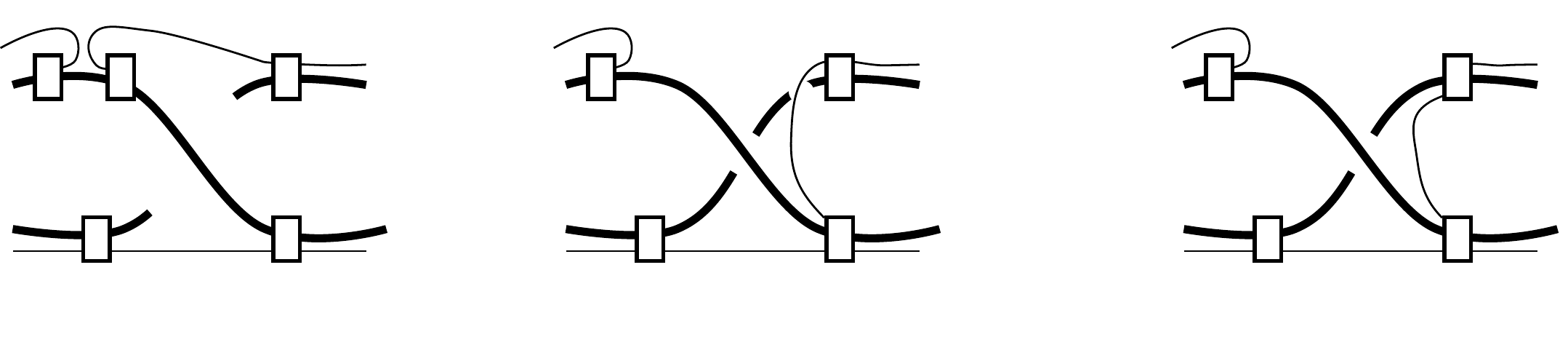
\end{figure} 

The number of crossings is reduced in the final chain complex, say $D_e$, by $n$, therefore $c(D_e) = c(D_i) - n$. Lemma \ref{lem:hoestimate} says that $|D_e| \geq -\frac{1}{2}c(D_e)$. Therefore, $\Kh_{i, *}(D_c) = 0$ for $i \leq (-\frac{1}{2}c(D_e))-1+1 = (-\frac{1}{2}c(D_i)+n)$ (also accounting for the degree shift of $\mathbf{h}^{\frac{1}{2}n}$ on $D_e$ and the shift of $\mathbf{h}$ on $D_c$.) This implies by Theorem \ref{thm:hles}, that $\widetilde{\Kh}(D_f)\stackrel{g_{II}}{\rightarrow} \widetilde{\Kh}(D_i)$ is an isomorphism for $i<n+1-2 = n-1$. Note that we do not need to consider other possible intersections of the single strand with the rest of the diagram as Rozansky does in his proof, since in the skein $\Sk^k$ there is only one $n$-cabled crossing. 
\end{enumerate}

For the last statement of the lemma concerning $\tilde{\Sk}^k_{I}$, we apply the exact same steps to the lower-left projector of $\widetilde{\Sk}^k_I$. 
\end{proof}  

\subsection{Local transformation III}
We first expand upon the mapping cone presentations of the Jones-Wenzl projector. The next few statements are taken directly from \cite{Roz12} with minor adjustments for our convention in crossing changes and reorganized. The proofs again are completely local, therefore, the reader can consult \cite{Roz12} for details.
\begin{thm}{\cite[Theorem 3.9]{Roz12}} \label{lem:lcm3he1}
There is a homotopy equivalence
\begin{figure}[H]
    \centering
    \def\svgwidth{.8\columnwidth}
    \begin{equation}
    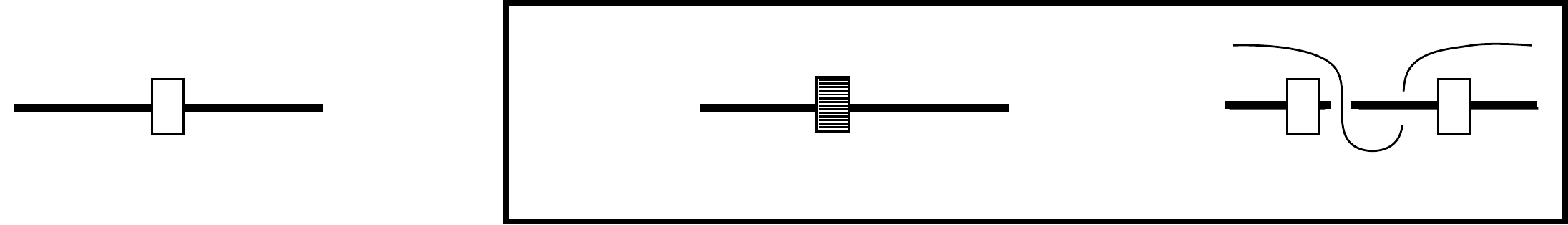
    \end{equation}
\end{figure}
\end{thm}

\begin{lem}{\cite[Lemma 3.12]{Roz12}} \label{lem:lcm3he2}
There is a homotopy equivalence
\begin{figure}[H]
\centering
    \def\svgwidth{.8\columnwidth}
    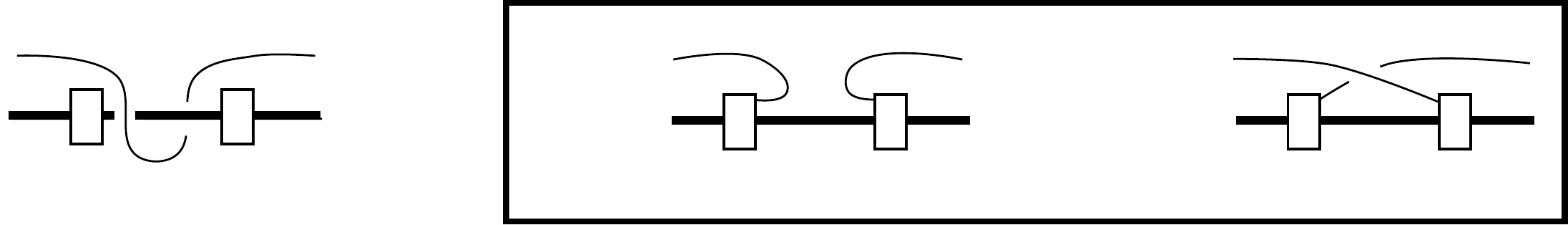
\end{figure}
\end{lem}

Combining Lemma \ref{lem:lcm3he1} and Lemma \ref{lem:lcm3he2} gives the following homotopy equivalence.

\begin{figure}[H]
\begin{equation}
\label{eq:lemcom}
\centering
    \def\svgwidth{1.1\columnwidth}
    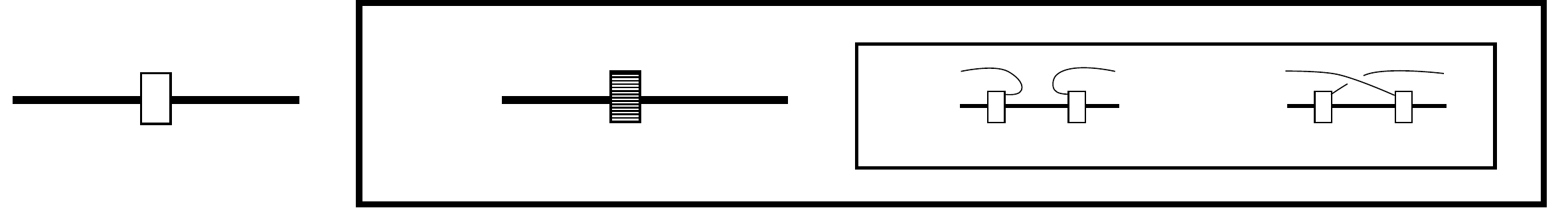
    \end{equation}
\end{figure}

Let $\mathbf{A}, \mathbf{B}$, and $\mathbf{C}$ denote the chain complexes   in the above mapping cone presentation. By Lemma \ref{lem:coneswitch}, if we let 

\begin{figure}[H]
\centering
    \def\svgwidth{.85\columnwidth}
    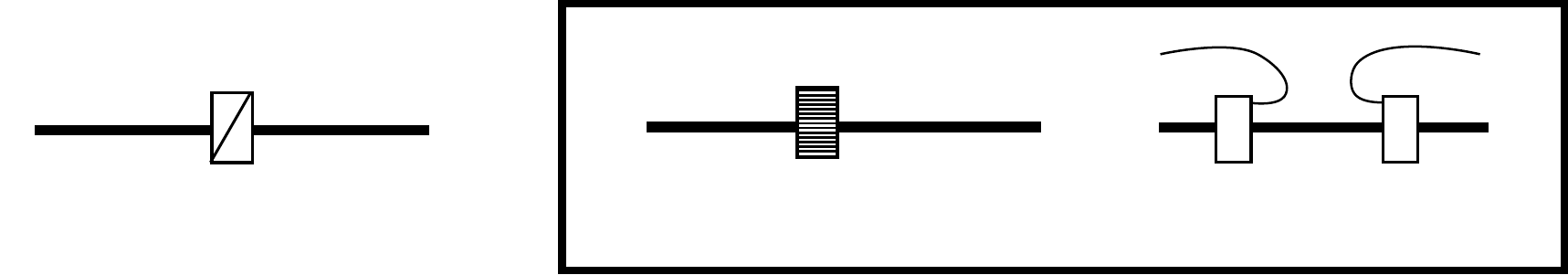
\end{figure}

be the complex $Cone(A[-1] \rightarrow B)$ resulting from writing the mapping cone presentation $Cone(A \rightarrow Cone(B \rightarrow C))$ of \eqref{eq:lemcom} as $Cone(Cone(A[-1] \rightarrow B) \rightarrow C)$, then we have the following theorem. 

\begin{thm}{\cite[Theorem 3.11]{Roz12}}\label{thm:lcm3he} The Jones-Wenzl projector has the following mapping cone presentation. 
\begin{figure}[H]
    \centering
    \def\svgwidth{.8\columnwidth}
    \begin{equation}
    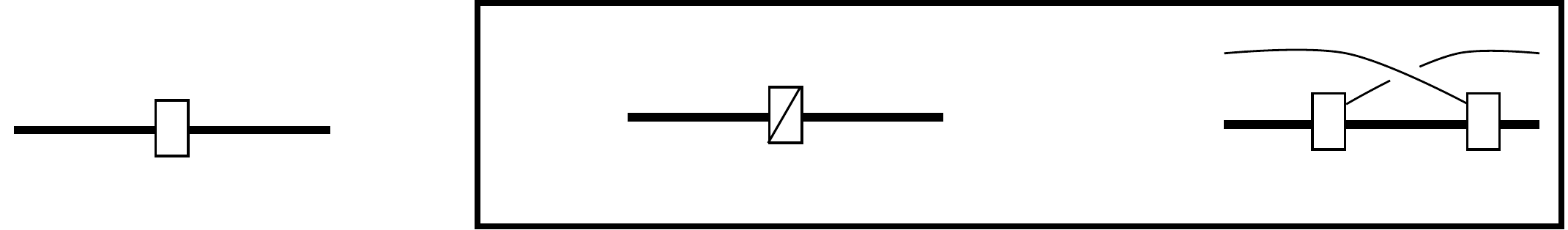
    \end{equation}
\end{figure}
\end{thm}

Now we can finally specify local transformation III.

\begin{prop}{\cite[Proposition 4.5]{Roz12}}
Let $\Sk^k_{II}$ be the skein obtained from $\Sk_{I}$ by two applications of the local transformation $M_{II}$ as in Proposition \ref{p:lcm2}, where there is a single line circle attached to the lower-right projector and a single line circle attached to the upper-left projector, with $(D_i, D_c, D_f)$ as shown in Figure \ref{fig:gformlcm3}.

\begin{figure}[H]
\centering
\def\svgwidth{\columnwidth}
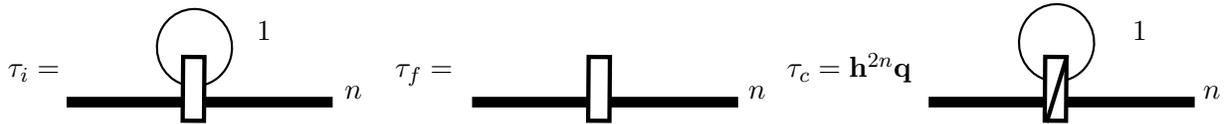
\caption{\label{fig:localskein3} Local transformation $M_{III}$.}
\end{figure}

Then
\begin{enumerate}[(a)]
\item there is a map $f_{III}$ such that we have a mapping cone presentation 
\[\tau_i = \fbox{$\tau_c \stackrel{f_{III}}{\rightarrow} \tau_f'$},\]
where $\tau_f' \sim \mathbf{q}^{-1}\tau_f$, and
\item \label{p:lcm2b} the induced degree-preserving map
\[ \widetilde{\Kh}(D_f) \stackrel{g_{III}}{\rightarrow} \widetilde{\Kh}(D_i) \] is an isomorphism on $\widetilde{\Kh_{i, *}}$ for $i\leq 2n-2$.

\begin{figure}[H]
\centering
\def\svgwidth{1.0\columnwidth}
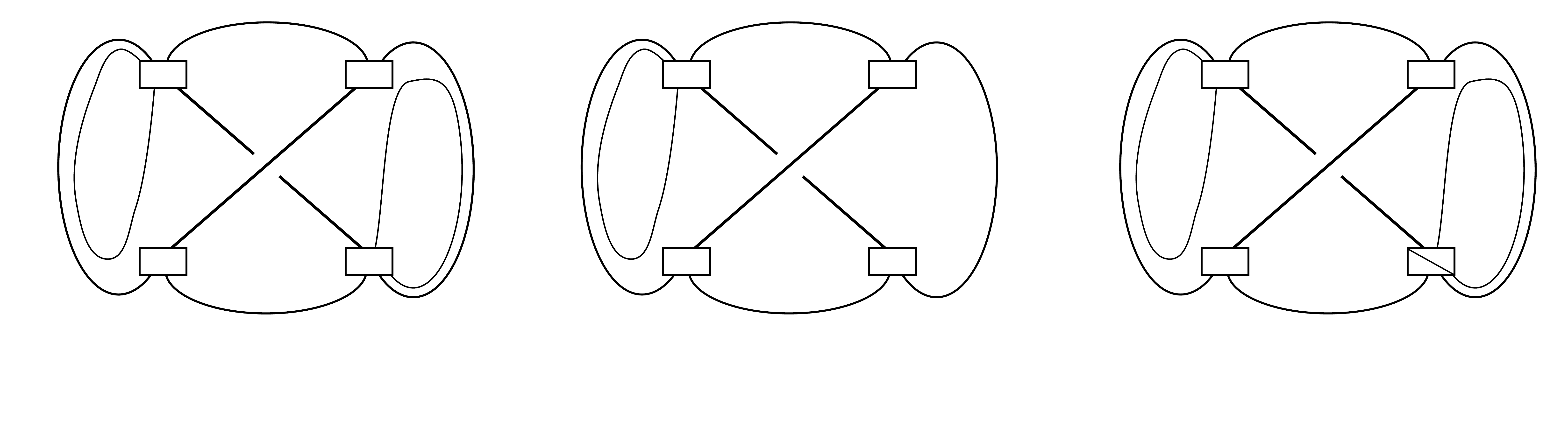
\caption{\label{fig:gformlcm3}}
\end{figure} 

\end{enumerate} 
Let $\widetilde{\Sk}^k_{II}$ be the skein $D_f$, and now  let $D_i =\widetilde{\Sk}^k_{II}$, we have the same statement for the triple $M_{III}=(\tau_i, \tau_f, \tau_c)$ applied to the upper-left projector with the remaining single-line circle attached.  

\begin{figure}[H]
\centering
\def\svgwidth{1.1\columnwidth}
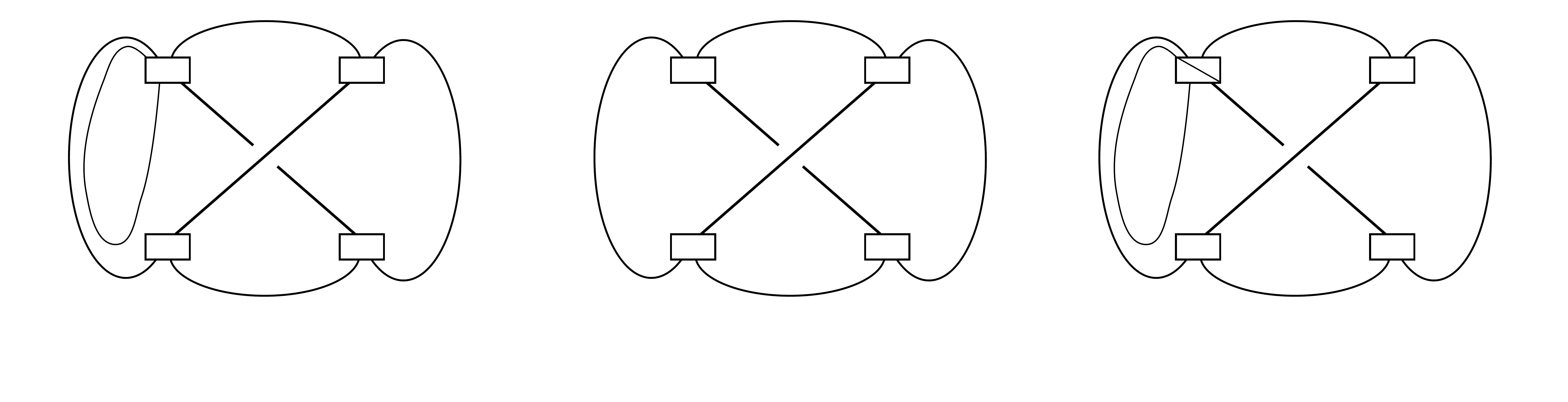
\end{figure} 

\end{prop}

\begin{proof}
\begin{enumerate}[(a)] 
\item From Theorem \ref{thm:lcm3he} we perform a type I Reidemeister move to complete the proof of part (a), see below.

\begin{figure}[ht]
\centering
    \def\svgwidth{.65\columnwidth}
    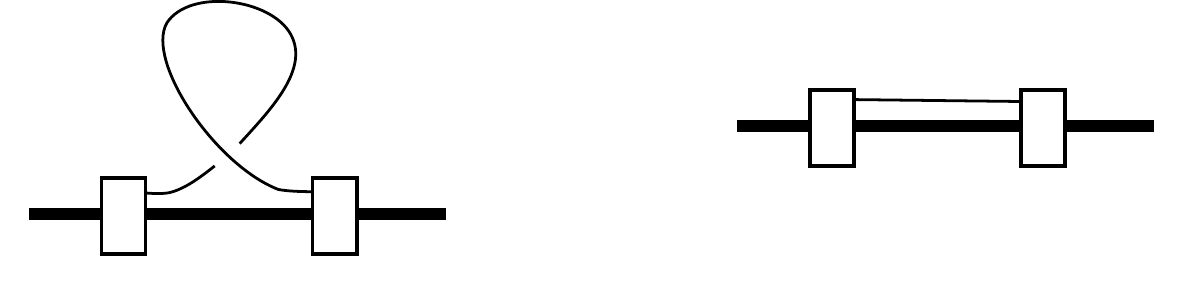
\end{figure}
Note that $c(D_i) =  c(D_f)$ and $|s_A(D_i)|-1=|s_A(D_f)|$.

\item It is clear that $D_c$ has the same number of crossings as $D_i$, so $|D_c| \geq 2n-\frac{1}{2}c(D_c) = 2n-\frac{1}{2}c(D_i)$, taking into account the degree shift $\mathbf{h}^{2n}\mathbf{q}$. Theorem \ref{thm:hles} then says that since $\Kh_{i, *}(D_c)=0$ for $i < 2n-\frac{1}{2}c(D_i)$, the map $g_{III}$ is an isomorphism on $\widetilde{\Kh_{i, *}}(D_f)$ for $i \leq 2n-\frac{1}{2}c(D_i) +\frac{1}{2}c(D_i)-2 = 2n-2$.

\end{enumerate} 
\end{proof} 

\subsection{Simplification of $\Sk^k$ via local transformations} \label{subsec:simpliflysk}
We apply the machinery set up in the previous section to understand the first few homology groups of the complex $\db{\Sk^k}$. The goal is to repeatedly apply local transformations $M_I, M_{II}$, and $M_{III}$ until we get to the cabled unknot.  

\begin{figure}[H]
\centering
    \def\svgwidth{1.1\columnwidth}
    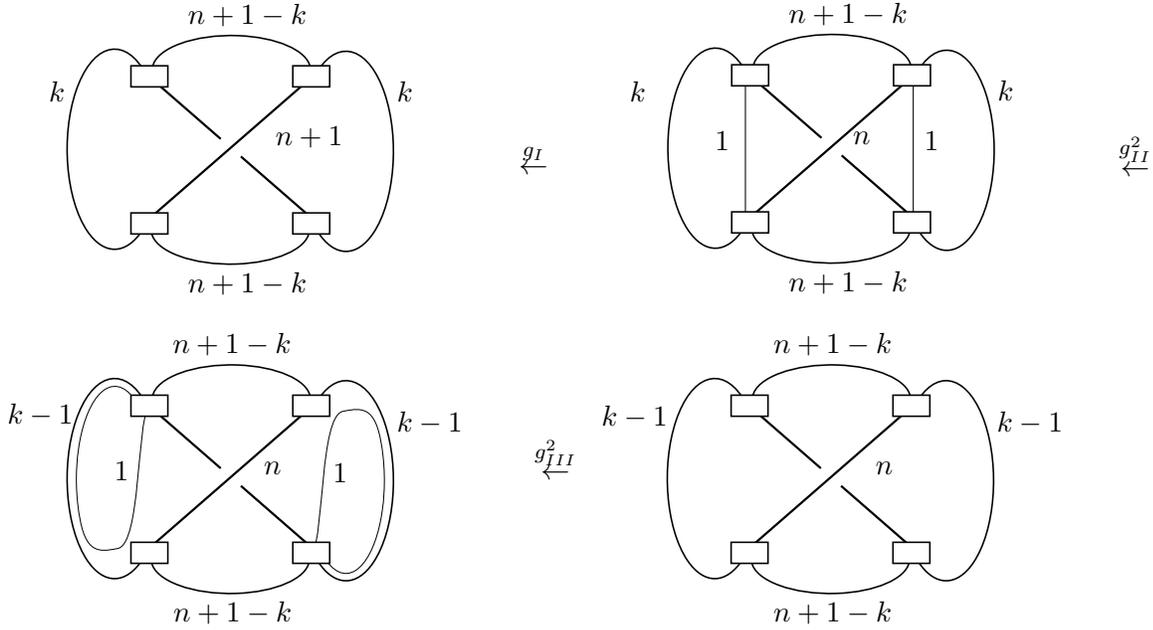
    \caption{\label{fig:gformpanel} The sequence of local transformations reducing $k$.}
\end{figure}

Let $\Sk\in \Sk(\mathbb{R}^2)$ be a skein decorated by Jones-Wenzl projectors, and let $\overline{\Sk}$ be the skein obtained by replacing the Jones-Wenzl projectors appearing in $\Sk$ by the identity. The shifted homology of $\Sk$ is defined as 
\[\widetilde{\Kh}(\Sk)= \mathbf{h}^{\frac{1}{2}c(\Sk)} \mathbf{q}^{|s_A(\overline{\Sk})|} \Kh(\Sk). \]

\begin{thm} \label{thm:dpuknot} Let $U^{n+1-k}$ be the $(n+1-k)$-cabled diagram of the standard diagram of the unknot with a left-hand twist, denoted by $U$, decorated by a Jones-Wenzl projector, see Figure \ref{fig:unknot}. There is a degree-preserving map
\[\widetilde{\db{\Sk^k}} \rightarrow \widetilde{\db{U^{n+1-k}}},  \] which is an isomorphism for $0 \leq i\leq n-k$. 
\begin{figure}[H]
\def\svgwidth{.2\columnwidth}
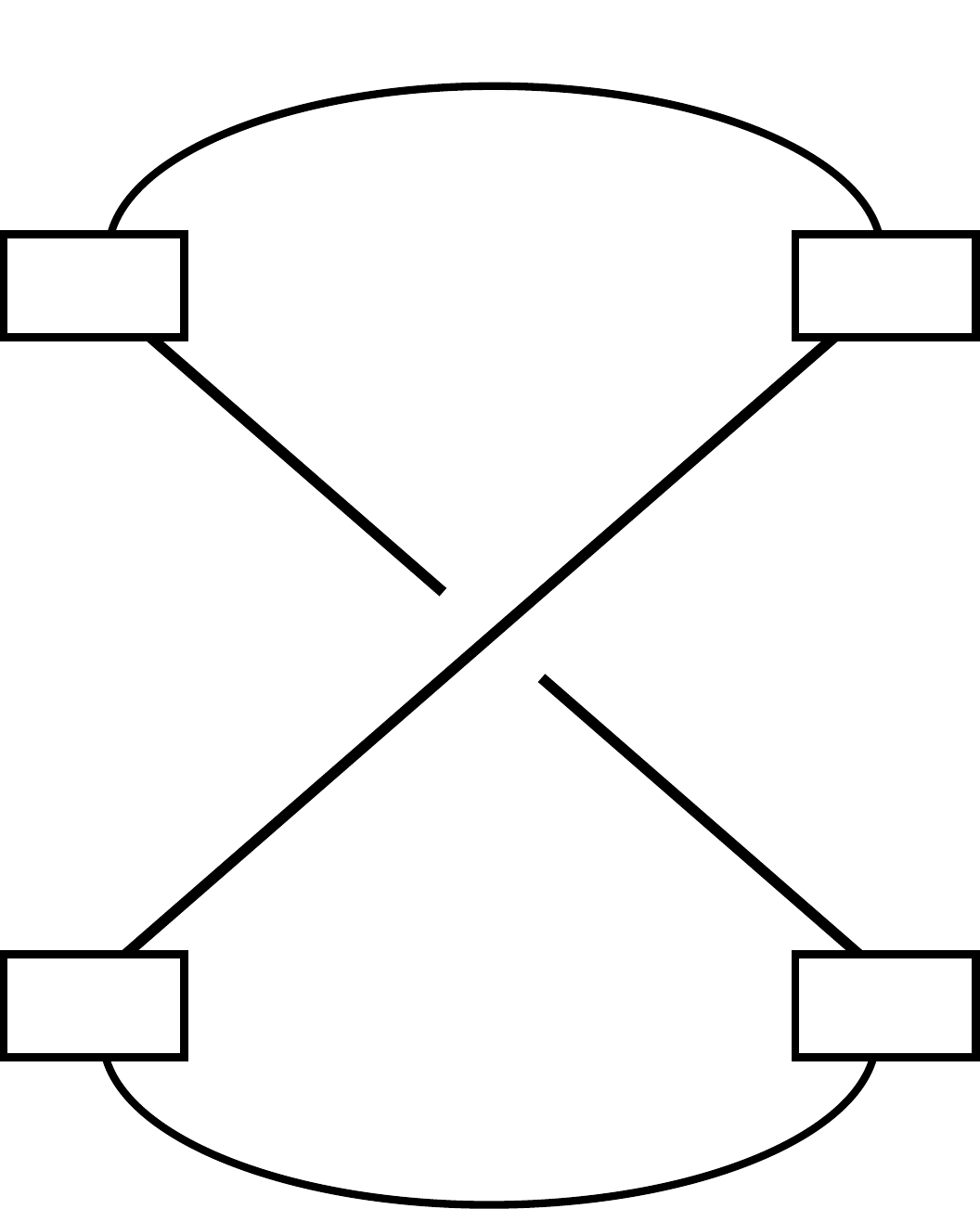
\caption{\label{fig:unknot} $U^{n+1-k}$}
\end{figure} 
\end{thm}
\begin{proof}
There is nothing to prove if $k = 0$, since then $\Sk^0 = U^{n+1}$. For $k\geq 1$, as shown in Figure \ref{fig:gformpanel} above, we may apply one $M_{I}$ move, then two $M_{II}$ moves, followed by two $M_{III}$ moves to $\Sk^{k}$, to obtain a skein with $k$ reduced to $k-1$ and an $n$-cabled crossing. The composition of the induced degree-preserving map $g_n=   g_I \circ g^2_{II} \circ g^2_{III} $ is then an isomorphism for $i \leq n-1$ by Theorem \ref{thm:hles}. We apply the same sequence of steps to the resulting skein in order to reduce $k$.  We similarly obtain a degree-preserving map $g_{n-1}$ by composition, except that this is now an isomorphism for $i< n-2$. Therefore, if we apply the sequence of steps $k$ times, we end up with a degree-preserving map $g_{n-k}$ which is an isomorphism for $0\leq i \leq n-k$. 
\end{proof} 

\subsection{Trivial homology of the unknot} \ \\

\begin{lem} \label{lem:unknottrivial} Let $0\leq k < n+1$. The homology group $\widetilde{\Kh_{i, *}}(U^{n+1-k})=0$ whenever $0\leq i \leq n-k$.   
\end{lem} 

\begin{proof}
Since we just have $U^{n+1-k}$ with a full left-hand twist we can remove the twist with a degree shift of $\mathbf{h}^{\frac{1}{2}a^2}\mathbf{q}^{a}$ for $a = n+1-k$. Note that $\frac{1}{2}a^2 > a-1$ for $a >0$.
\end{proof}

\subsection{Putting everything together}

By Corollary \ref{cor:directlimit}, it suffices to show that the homology group $\widetilde{\Kh_{n, *}}(D, n+1)$ is trivial. We will show that $\widetilde{\Kh_{i, *}}(D, n+1)$ is trivial for $0\leq i\leq n$.

\begin{thm} \label{cor:tail-trivial} The homology groups $\widetilde{\Kh_{i, *}}(D, n+1)$ are trivial for $i\leq n$.
\end{thm}

\begin{proof} 
Recall that $D$ is a non-$A$ adequate diagram and $\ell$ is a crossing in $D$ whose corresponding segment in the all-$A$ state is a one-edged loop. 
We may decompose the complex $\db{D^{n+1}_{\vcenter{\hbox{\def\svgwidth{.020\columnwidth} 
\begingroup%
  \makeatletter%
  \providecommand\color[2][]{%
    \errmessage{(Inkscape) Color is used for the text in Inkscape, but the package 'color.sty' is not loaded}%
    \renewcommand\color[2][]{}%
  }%
  \providecommand\transparent[1]{%
    \errmessage{(Inkscape) Transparency is used (non-zero) for the text in Inkscape, but the package 'transparent.sty' is not loaded}%
    \renewcommand\transparent[1]{}%
  }%
  \providecommand\rotatebox[2]{#2}%
  \ifx\svgwidth\undefined%
    \setlength{\unitlength}{69.35bp}%
    \ifx\svgscale\undefined%
      \relax%
    \else%
      \setlength{\unitlength}{\unitlength * \real{\svgscale}}%
    \fi%
  \else%
    \setlength{\unitlength}{\svgwidth}%
  \fi%
  \global\let\svgwidth\undefined%
  \global\let\svgscale\undefined%
  \makeatother%
  \begin{picture}(1,1.38428262)%
    \put(0,0){\includegraphics[width=\unitlength]{jwproj.pdf}}%
  \end{picture}%
\endgroup%
}}}}$ over Kauffman states restricted to the set of crossings $C(D^{n+1}) \setminus \ell^{n+1}$, where $\ell^{n+1}$ is the set of $(n+1)^2$ crossings corresponding to $\ell$ in $D^{n+1}$. On the level of chain groups, $\Ckh(D, n+1)$ may be written as a direct sum of the chain groups $\Ckh(\Sk_{\sigma})$ with the appropriate grading shifts: 
\begin{align*} 
\Ckh(D, n+1) &= \bigoplus_{\sigma \text{ a Kauffman state}} \mathbf{h}^{\sgn(\sigma)} \Ckh(\Sk_{\sigma}),\\
\intertext{and by Lemma \ref{lem:disjointsk}, $\Sk_{\sigma}$ is a disjoint union of $\Sk^{k}_{\sigma}$ with circles.}
&= \bigoplus_{\sigma \text{ a Kauffman state}} \mathbf{h}^{\sgn(\sigma)} \Ckh(\Sk^k_{\sigma} \sqcup \text{ disjoint circles}).
\end{align*}

Let $\sigma_A$ be the all-$A$ state which chooses the $A$-resolution on every crossing in $C(D^{n+1}) \setminus \ell^{n+1}$, and let $c = \sgn(\sigma) - \sgn(\sigma_A)$. 
Because $\Sk_{\sigma} = \Sk^k_{\sigma} \ \sqcup$ disjoint circles, if $\widetilde{\Kh_{i-c, *}}(\Sk^k_{\sigma}) = 0$, then $\widetilde{H^{Kh}_{i-c, *}}(\Sk_{\sigma})=0$ which implies $\widetilde{\Kh_{i, *}}(D, n+1)=0$ by Lemma \ref{lem:decompstate}. Thus it suffices to show for any $\sigma$ that
\[ \widetilde{\Kh_{i-c, *}}(\Sk^k_{\sigma}) = 0 \text{ for } i \leq n. \] 
For a fixed $\sigma$, if $i-c <0$, then there is nothing to prove since the homology groups with negative $i$-grading all vanish in the shifted homology. So the remaining case is where $i-c \geq 0$. Theorem \ref{cor:tail-trivial} will immediately follow from Theorem \ref{thm:dpuknot} and Lemma \ref{lem:unknottrivial} as long as we can show $c\geq k$. 

That is the content of the following lemma, where we use the fact that $\ell$ is a crossing corresponding to a one-edged loop in the all-$A$ state of $D$. 

\begin{lem} Let $\ell$ be a crossing in the non $A$-adequate diagram $D$ whose corresponding segment in $s_A(D)$ is a one-edged loop. For a skein $\Sk_{\sigma}^k$ containing $\ell^n$ as in Figure \ref{fig:gform} from a Kauffman state $\sigma$ on $C(D^n) \setminus \ell^n$, we have
$c=\sgn(\sigma)-\sgn(\sigma_A) \geq k$. 
\end{lem} 

\begin{proof}  There are two cases. A state $\sigma \not= \sigma_A$ with $k=0$ certainly has $c \geq k$. If $\sigma$ is a state for which $k>0$, then we argue that it must choose the $B$-resolution on at least $k$ crossings. 

Since the segment corresponding to $\ell$ in the all-$A$ state of $D$ is a one-edged loop, the skein $\Sk^k_A$ from the state $\sigma_A$ must have $k=0$. In other words, $\Sk^{k}_A = \Sk^{0}_A = U^{n+1}$ as in Figure \ref{fig:unknot}. Let $B_{\sigma}$ be the set of crossings in $C(D^n)\setminus \ell^n$ on which $\sigma$ chooses the $B$-resolution. We can obtain $\Sk_{\sigma}$ from $\Sk_{\sigma_A}$ by locally removing two arcs of the $A$-resolution and replacing them by the two arcs of the $B$-resolution at each segment in $\Sk_{\sigma_A}$ corresponding to a crossing in $B_{\sigma}$. See Figure \ref{fig:localrr} below for the local replacement. A state has a local replacement on a segment if it chooses the $B$-resolution at the crossing corresponding to the segment.

\begin{figure}[ht]
\centering
\def\svgwidth{.6\columnwidth}
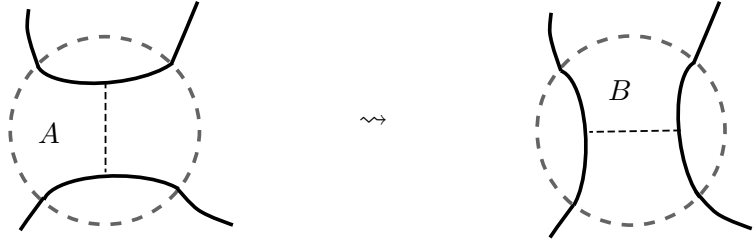
\caption{\label{fig:localrr}We change the state circles locally.}
\end{figure}

Number the outermost $k$ strands joining the top-right and the top-left projectors in $S^0_A$, outermost first, from $s_1, s_2, \ldots$, to $s_k$. Suppose that $|B_{\sigma}|< k$. There must be some $i\in 1, \ldots, k-1$ such that $\sigma$ has no local replacement from $A$ to $B$ on any segment between $s_i$ and $s_{i+1}$, or $\sigma$ has no local replacement from $A$ to $B$ on any segment attached to $s_1$ on the outside, see Figure \ref{fig:lessn} for the convention of what is meant by the ``outside" of the strand $s_1$ and ``between" the strands  $s_i$ and $s_{i+1}$. Either of these cases results in fewer than $2k$ strands joining the top and bottom pairs of projectors for $\Sk^k_{\sigma}$, a contradiction. This means that $|B_{\sigma}| \geq k$, and we have the statement of the lemma.

\begin{figure}[ht]
\centering
\def\svgwidth{.6\columnwidth}
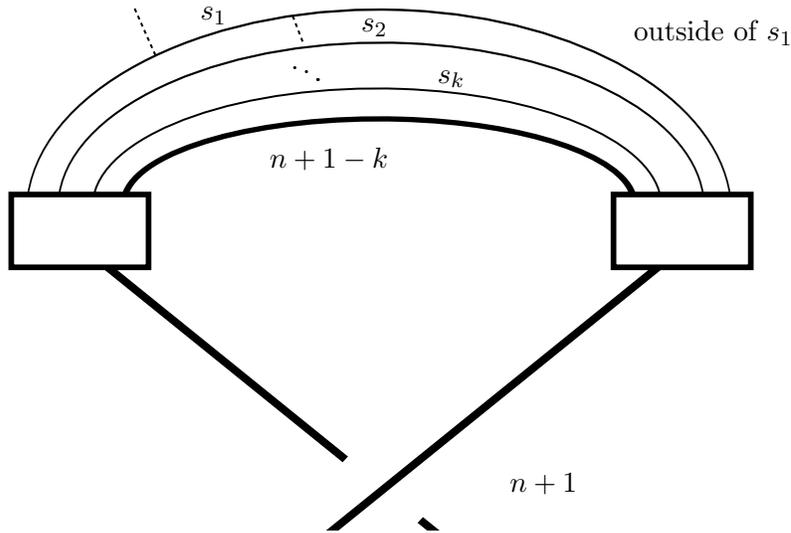
\caption{\label{fig:lessn} The strands $s_1, \ldots, s_k$ are marked here for $\Sk^0_A$. Two dashed segments are shown. One is attached to $s_1$ from the outside, and another is between $s_1$ and $s_2$. Segments which are not shown may attach to the strands in a complicated fashion. However, if $\sigma$ has no local replacement on any segment attached to $s_1$ from the outside, then $k=0$ for $\Sk^k_{\sigma}$. If $\sigma$ has no local replacement on any segment between $s_i$ and $s_{i+1}$, then there cannot be more than $2i$ strands which end up connecting the top and bottom pairs of projectors in $\Sk^k_{\sigma}$.}
\end{figure}

\end{proof}
\end{proof}

\bibliographystyle{amsalpha}
\bibliography{references}

\providecommand{\bysame}{\leavevmode\hbox to3em{\hrulefill}\thinspace}
\providecommand{\MR}{\relax\ifhmode\unskip\space\fi MR }
\providecommand{\MRhref}[2]{%
  \href{http://www.ams.org/mathscinet-getitem?mr=#1}{#2}
}
\providecommand{\href}[2]{#2}
\begin{thebibliography}{LvdV16}

\bibitem[AD11]{AD2}
Cody Armond and Oliver~T. Dasbach, \emph{Rogers--{R}amanujan type identities
  and the head and tail of the colored {J}ones polynomial}, arXiv:1106.3948,
  2011.

\bibitem[Arm13]{Arm}
Cody Armond, \emph{The head and tail conjecture for alternating knots},
  Algebraic and Geometric Topology \textbf{13} (2013), 2809--2826.

\bibitem[BN05]{Bar05}
Dror Bar-Natan, \emph{Khovanov's homology for tangles and cobordisms}, Geometry
  and Topology \textbf{9} (2005), 1443--1499. \MR{2174270}

\bibitem[CK12]{CK12}
Benjamin Cooper and Vyacheslav Krushkal, \emph{Categorification of the
  {J}ones-{W}enzl projectors}, Quantum Topology \textbf{3} (2012), no.~2,
  139--180. \MR{2901969}

\bibitem[Cos14]{Cos14}
Francesco Costantino, \emph{Integrality of {K}auffman brackets of trivalent
  graphs}, Quantum Topology \textbf{5} (2014), no.~2, 143--184.

\bibitem[DL06]{DL06}
Oliver~T. Dasbach and Xiao-Song Lin, \emph{On the head and the tail of the
  colored {J}ones polynomial}, Compositio Mathematica \textbf{142} (2006),
  no.~5, 1332--1342.

\bibitem[FKS06]{FKS06}
Igor Frenkel, Mikhail Khovanov, and Catharina Stroppel, \emph{A
  categorification of finite-dimensional irreducible representations of quantum
  {$\mathfrak{sl}_2$} and their tensor products}, Selecta Mathematica. New
  Series \textbf{12} (2006), no.~3-4, 379--431. \MR{2305608}

\bibitem[GL15]{GL15}
Stavros Garoufalidis and Thang T.~Q. L{\^e}, \emph{Nahm sums, stability and the
  colored {J}ones polynomial}, Research in the Mathematical Sciences \textbf{2}
  (2015), 1--55.

\bibitem[Lee14]{Lee14}
Christine Ruey~Shan Lee, \emph{Stability properties of the color {J}ones
  polynomial}, arXiv:1409.4457, 2014.

\bibitem[Lic97]{Lic97}
William B.~R. Lickorish, \emph{An introduction to knot theory}, Graduate Texts
  in Mathematics, vol. 175, Springer-Verlag, New York, 1997. \MR{1472978}

\bibitem[LT88]{LT88}
William B.~R. Lickorish and Morwen Thistlethwaite, \emph{Some links with
  non-trivial polynomials and their crossing-numbers}, Commentarii mathematici
  Helvetici \textbf{63} (1988), no.~4, 527--539.

\bibitem[LvdV16]{LV}
Christine Ruey~Shan Lee and Roland van~der Veen, \emph{Slopes for pretzel
  knots}, New York Journal of Mathematics \textbf{22} (2016), 1339--1364.

\bibitem[Roz14a]{Roz10}
Lev Rozansky, \emph{An infinite torus braid yields a categorified
  {J}ones-{W}enzl projector}, Fundamenta Mathematica \textbf{225} (2014),
  no.~1.

\bibitem[Roz14b]{Roz12}
\bysame, \emph{Khovanov homology of a unicolored {B}-adequate link has a tail},
  Quantum Topology \textbf{5} (2014), no.~4, 541--579.

\bibitem[RT90]{RT90}
N.~Yu. Reshetikhin and V.~G. Turaev, \emph{Ribbon graphs and their invariants
  derived from quantum groups}, Communications in Mathematical Physics
  \textbf{127} (1990), no.~1, 1--26. \MR{1036112}

\bibitem[Wen87]{Wen87}
Hans Wenzl, \emph{On sequences of projections}, C.R. Math. Rep. Acad. Sci.
  Canada \textbf{9} (1987), no.~1, 5--9. \MR{MR873400}

\bibitem[Yam92]{Yam89}
Shuji Yamada, \emph{A topological invariant of spatial regular graphs}, Knots
  90 ({O}saka, 1990), de Gruyter, Berlin, 1992, pp.~447--454. \MR{1177441}

\end{thebibliography}

\end{document}